\documentclass[a4paper, 10pt]{article}

\usepackage{graphicx}
\usepackage{amsmath,amssymb,amsfonts,amsthm}
\usepackage{mathtools}
\usepackage{cite}
\usepackage{hyperref}
\usepackage{xifthen}
\usepackage{authblk}
\usepackage{enumitem}			
\usepackage{cleveref}
\usepackage{pgfplots} 				
\usepackage[margin=3cm]{geometry}
\graphicspath{{./Figures/}}

\newcommand{\R}{\mathbb{R}}
\newcommand{\N}{\mathbb{N}}
\newcommand{\Z}{\mathbb{Z}}
\newcommand{\C}{\mathbb{C}}
\newcommand{\inner}[2]{\ifthenelse{\equal{#2}{}}{\left\langle\cdot,\cdot\right\rangle_{#1}}{\left\langle#2\right\rangle_{#1}}}
\newcommand{\norm}[2]{\ifthenelse{\equal{#2}{}}{\left\|\cdot\right\|_{#1}}{\left\|#2\right\|_{#1}}}
\newcommand{\seminorm}[2]{\ifthenelse{\equal{#2}{}}{\left|\cdot\right|_{#1}}{\left|#2\right|_{#1}}}
\newcommand{\ns}{\mathcal H(\Omega)}
\newcommand{\calh}{\mathcal H}
\newcommand{\fa}{\hbox{ for all }}
\DeclareMathOperator*{\Sp}{span}
\newcommand{\call}{\mathcal L}
\newtheorem{theorem}{Theorem}
\newtheorem{prop}[theorem]{Proposition}
\newtheorem{cor}[theorem]{Corollary}
\newtheorem{lemma}[theorem]{Lemma}
 
\newtheorem{rem}[theorem]{Remark}
\newtheorem{example}[theorem]{Example}
\newtheorem{assumption}[theorem]{Assumption}
\newcommand{\Ht}{\mathcal{H}_{\theta}}

\title{General superconvergence for kernel-based approximation}

\author[1]{Toni Karvonen\thanks{\href{toni.karvonen@lut.fi}{toni.karvonen@lut.fi}}} 
\affil[1]{School of Engineering Sciences, LUT University (Lappeenranta, Finland)}

\author[2]{Gabriele Santin\thanks{\href{mailto:gabriele.santin@unive.it}{gabriele.santin@unive.it}}} 
\affil[2]{Department of Environmental Sciences, Informatics and Statistics, Ca' Foscari University of Venice (Venice, Italy)}

\author[3,4]{Tizian Wenzel\thanks{\href{mailto:wenzel@math.lmu.de}{wenzel@math.lmu.de}}}
\affil[3]{Department of Mathematics, Ludwig Maximilian University of Munich (Munich, Germany)}
\affil[4]{Munich Center for Machine Learning (Munich, Germany)}

\date{\today}

\begin{document}
\maketitle 

\begin{abstract}
Kernel interpolation is a fundamental technique for approximating functions from scattered data, with a well-understood convergence theory when interpolating 
elements of a reproducing kernel Hilbert space. 
Beyond this classical setting, research has focused on two regimes: misspecified interpolation, where the kernel smoothness exceeds that of the target 
function, and superconvergence, where the target is smoother than the Hilbert space. 

This work addresses the latter, where smoother target functions yield improved convergence rates, and extends existing results by characterizing 
superconvergence for projections in general Hilbert spaces. 
We show that functions lying in ranges of certain operators, including adjoint of embeddings, 
exhibit accelerated convergence, which we extend across interpolation scales 
between these ranges and the full Hilbert space. 
In particular, we analyze Mercer operators and embeddings into $L_p$ spaces, linking the images of adjoint operators to Mercer power spaces.  
Applications to Sobolev spaces are discussed in detail, highlighting how superconvergence depends critically on boundary conditions. 
Our findings generalize and refine previous results, offering a broader framework for understanding and exploiting superconvergence.
The results are supported by numerical experiments.
\end{abstract}

\section{Introduction}\label{sec:intro}

Kernel interpolation is a well established and efficient technique for the approximation of functions from samples on a finite set of scattered points in 
possibly high dimension~\cite{wendland2005scattered,fasshauer2015kernel}.

Each positive definite kernel is uniquely associated to a Reproducing Kernel Hilbert Space (RKHS) $\calh$, or native space, where it acts as a reproducing 
kernel, 
meaning that the kernel function is the Riesz representer of the point evaluation functional \cite{saitoh2016theory}. 
Any Hilbert space of functions with continuous point evaluations is in fact the RKHS of a given kernel, and for functions inside these Hilbert spaces kernel 
interpolation coincides with an orthogonal projection, and the corresponding convergence theory is well 
understood. 

An instance of paramount importance in applied mathematics is the Sobolev space $W_2^\tau(\Omega)$ of possibly fractional order $\tau>d/2$ on a bounded set 
$\Omega\subset\R^d$ of sufficient regularity. 
In this case, kernel interpolation on well distributed points $X\subset\Omega$ achieves worst-case optimal 
approximation rates~\cite{wendland2005scattered}. For example, $n\in\N$ asymptotically uniformly distributed points give an $L_2$ 
convergence of the interpolation error of order $n^{-\tau/d}$, and similar results hold for any other $W_q^m$ norm, $0\leq m\leq \tau$, with appropriately 
modified rates of convergence~\cite{narcowich2004zeros}.
For these spaces, there are several explicitly known kernels that can be used for computing the interpolant in practice, such as the families of 
Mat\`ern and Wendland kernels \cite{porcu2024matern,wendland1995piecewise} whose RKHSs are norm equivalent to Sobolev spaces. 

For general kernels and RKHSs, it is natural to investigate how these optimal approximation properties extend to functions which are not in 
$\calh$ or, on the other extreme, which are elements of some special subset of $\calh$. 
In the Sobolev setting this situation has the practical relevance of modeling the case when the smoothness of the chosen kernel does not match that of the 
target function, and is either higher or lower.
These two cases have been studied in the literature under different names and points of view. 

The case of a rough function being approximated with a smoother kernel is known as the \emph{misspecified 
setting}~\cite{kanagawa2016convergence,kanagawa2019convergence}, and resulting error bounds allow one to \emph{escape the native space}. 
For Sobolev kernels in 
particular, under standard assumptions it is known that optimal rates can be obtained also in this setting under an additional stability condition. Namely, 
using well separated points one can approximate a function in 
$W_2^\beta$ with a $W_2^\tau$ kernel, $\tau\geq \beta>d/2$, and the resulting $W_2^\mu$ rates of convergence, $\mu\leq \beta$, match those which would be 
obtained with the correct $W_2^\beta$ kernel (Theorem 4.2 in~\cite{narcowich2006sobolev}). Recently, also sharp inverse statements have been proven in this 
regime~\cite{wenzel2025sharp}, showing that one can infer the smoothness of the target function by the observed rates of decay of the $L_2$ interpolation error.

The other direction of research, which is the one that we address in this work, investigates instead the case of a function in $\calh$ that is 
smoother than the whole RKHS, and one would like to determine to what extent this property is reflected in improved convergence rates. This 
\emph{superconvergence setting} has been initially 
investigated in~\cite{schaback1999improved} by proving that the order of convergence of the $L_2$ error is doubled when interpolating functions that are in the 
range of the Mercer operator $T:L_2\to L_2$ which integrates a function against the kernel. This operator can be shown to be the adjoint of the 
embedding operator of $\calh$ into $L_2$, and can be defined under fairly general assumptions~\cite{steinwart2012mercer}.
This initial result has been significantly re-examined and extended recently by Schaback~\cite{schaback2018superconvergence}, who considers images of general embedding operators and error in norms other than $L_2$.
He discusses in particular the role of localization, proving that often 
superconvergence requires the target function not only to be smoother, but also to meet some specific boundary conditions, which are usually hidden and not 
explicitly available. 
Moreover, superconvergence is obtained by showing that for these subspaces of special functions one can bound the 
interpolation error in the $\calh$ norm with the same rate as the $L_2$ error that holds for a general function of the space. This is in contrast 
to the worst-case in $\calh$, where the interpolation error is merely bounded, but generally not convergent (see Chapter 8 in~\cite{iske2019book} and 
Corollary 2.4 in~\cite{karvonen2022error}).
Furthermore, Sloan and Kaarnioja~\cite{sloan2025doubling} have extended this doubling trick to the case of 
general Hilbert spaces and approximation by orthogonal
projection (and not necessarily interpolation) into suitable finite subspaces and again in the $L_2$ norm. These subspaces are characterized by a specific 
boundedness condition (equation (2) in~\cite{sloan2025doubling}), that is the same that was used in~\cite{santin2021sampling} to show convergence of 
kernel-based quadrature, which requires in turn the $\calh$-convergence of kernel interpolation. 
In~\cite{sloan2025doubling}, it is further observed that instances of these subspaces arise as subspaces of certain Sobolev spaces, characterized by 
additional smoothness and certain boundary conditions.
For Sobolev spaces, the work~\cite{hangelbroek2025extending} has extended the doubling of convergence rates to the error measured in higher-order Sobolev 
norms. 
Moreover, both~\cite{sloan2025doubling} and~\cite{hangelbroek2025extending} work for certain conditionally positive definite kernels.
We mention additionally that~\cite{schaback2018superconvergence} contains a detailed account of superconvergence results in approximation theory for other type 
approximation 
methods. Other instances appear for example in statistical learning theory, see e.g.~\cite{blanchard2018optimal,li2024saturation}, and Theorem~1 
in~\cite{FischerSteinwart2020}. 
Moreover, as we are partially analyzing the connection between symmetric positive definite kernels and Green kernels, we stress that this topic is the 
subject of a line of research \cite{fasshauer2012green,fasshauer2011reproducing,fasshauer2013reproducing} constructing Green kernels and generalized Sobolev 
spaces (see also 
\cite[Chapters 6--8]{fasshauer2015kernel}). For these kernels, corresponding superconvergence results, even of fractional order, have recently 
appeared~\cite{mohebalizadeh2022refined}.

In this paper we partially extend these works in several directions. After providing some background material in~\Cref{sec:background}, we initially consider 
in~\Cref{sec:cont_superconv} the approximation by orthogonal projections in a general Hilbert space $\calh$, and prove superconvergence for functions that are 
in the range of (the adjoint of) a fairly general linear and continuous operator mapping from $\calh$ to a general Banach space (\Cref{th:super_A_star}), 
including embeddings 
as a special case.
Moreover, we are able to extend these results to the whole scale of real interpolation spaces between these ranges and the full $\calh$ 
(\Cref{cor:intermediate}). This gives rise to superconvergence that is not limited to double rates, but rather to the entire scale of intermediate speeds of convergence. 

For an RKHS and the classical case of the Mercer operator, we show in~\Cref{sec:superconv_mercer} that this class of intermediate spaces comprises the 
full scale of \emph{power spaces}, namely, those obtained by suitable summability conditions with respect to the Mercer eigenbasis (\Cref{prop:theta_spaces}).
We then consider in~\Cref{sec:superconv_in_images} the more general case of operators mapping to $L_p$, showing in particular that their adjoints are still 
integral operators but with a modified kernel (\Cref{prop:A_as_integral_operator}), and provide a specialized superconvergence result in \Cref{cor:integral_op}. 
Moreover, we analyze in detail the case of the embedding operators of $\calh$ in $L_p$ and provide a strict relation between the image of their adjoint and 
the Mercer power spaces (\Cref{th:embedding_power_spaces}).

All these extensions are significant as they allow us to cover cases not comprised in previous results (see, e.g., \Cref{ex:l1,ex:zonal,ex:periodic_kernel}) 
and, especially in the Sobolev case, to obtain superconvergence under more general conditions and for the a full scale of Sobolev spaces 
(\Cref{cor:application_to_sobolev} and \Cref{cor:TL_superconvergence-Sobolev}).
Moreover, we quantify the error in a general $W_q^m$ norm $m<\tau$. 
As an exemplary case, \Cref{cor:application_to_sobolev} shows that functions that are in a certain real interpolation space between
$\calh(\Omega)$ and $T(L_2(\Omega))$, when interpolated with $n$ well-distributed points, are approximated in
the $W_q^m$-norm with an error bounded by $n^{-(1+\theta)\tau/d + m/d + \left(1/2 -1/q\right)_+}$, where the parameter $\theta\in[0,1]$ continuously scales 
between the spaces $\calh(\Omega)$ and $T(L_2(\Omega))$.

Since these superconvergence conditions are expressed in terms of ranges of certain operators, or even as interpolation spaces, they may not be trivial to check 
in practice, as stressed in~\cite{schaback2018superconvergence}. 
To this end, in \Cref{sec:boundary} we discuss the notable case of the embedding operator of a Sobolev RKHS in $L_2$, and we provide a connection 
between superconvergence and the set of solutions of certain linear PDEs, and their Green kernels. 
The discussion in this section highlights the fact that two kernels with RKHSs that are norm equivalent to the same Sobolev space may give rise to different 
conditions for superconvergence. 
In particular, in contrast to the misspecified setting, one should not expect to obtain the correct rate 
of approximation by interpolating a $W_2^{\beta}$ function with a $W_2^{\tau}$ kernel, $\tau<\beta$, unless this function meets additional conditions.
We conclude the paper with several numerical experiments in \Cref{sec:experiments}, which illustrate and verify some aspects of our theory.

\section{Preliminaries}
\label{sec:background}

Let $\Omega$ be any set.
We consider a positive definite kernel $k \colon \Omega \times \Omega \to \R$ on $\Omega$.
This means that the kernel Gram matrix $(k(x_i, x_j))_{i,j=1}^n \in \R^{n \times n}$ is symmetric and positive semi-definite for any $n \in \N$ and any points 
$x_1, \ldots, x_n \in \Omega$.
That is, $\sum_{i=1}^n \sum_{j=1}^n a_i a_j k(x_i, x_j) \geq 0$ for any $n \in \N$, $x_1, \ldots, x_n \in \Omega$ and $a_1, \ldots, a_n \in \R$.
The kernel is strictly positive definite if this quadratic form is positive whenever not all $a_i$ are zero and $x_1, \ldots, x_n$ are pairwise distinct.

\subsection{RKHS and orthogonal projections}
By the Moore--Aronszajn theorem, every positive semi-definite kernel induces a unique \emph{reproducing kernel Hilbert space} (RKHS), also called 
\emph{native space}, which we denote $\calh(\Omega)$ or, if it is necessary to indicate the kernel, $\calh(k, \Omega)$.
The RKHS is a Hilbert space of functions on $\Omega$ in which the point evaluation functional $f \mapsto f(x)$ is continuous for every $x \in \Omega$.
The kernel $k$ reproducing point evaluations, in that $k(\cdot, x) \in \calh(\Omega)$ and
\begin{equation} \label{eq:reproducing-property}
  \langle f, k(\cdot, x) \rangle_{\calh(\Omega)} = f(x)
\end{equation}
for all $x \in \Omega$ and $f \in \calh(\Omega)$.
The elements of $\calh(\Omega)$ inherit many properties of the kernel.
For example, if $k$ is $m$ times continuously differentiable in each argument, then all elements of $\calh(\Omega)$ are $m$ times continuously 
differentiable~\cite[Theorem~2.6]{saitoh2016theory}.
By the Riesz representation theorem, the reproducing propery~\eqref{eq:reproducing-property} extends all functionals in the dual space $\calh(\Omega)'$, the space of continuous linear functionals on $\calh(\Omega)$. 
Namely, if $L \in \calh(\Omega)'$, then
\begin{equation*}
  L(f) = \langle f, k_L \rangle_{\calh(\Omega)}
\end{equation*}
for every $f \in \calh(\Omega)$, where the function $k_L \in \calh(\Omega)$ given by $k_L(x) = L(k(\cdot, x))$ is the representer of $L$.
We refer to \cite{Berlinet2004, saitoh2016theory, Paulsen2016} for an exhaustive theory of RKHS.

We consider an arbitrary $\calh(\Omega)$-orthogonal projection $P \colon \calh(\Omega) \to \calh(\Omega)$.
Recall that an orthogonal projection satisfies $P(Pf) = Pf$ and $\langle f - Pf, Pf \rangle_{\calh(\Omega)} = 0$ for all $f \in \calh(\Omega)$.
From the latter property it follows that \smash{$\lVert f - Pf \rVert_{\calh(\Omega)}^2 = \langle f - Pf, f \rangle_{\calh(\Omega)}$}.
Let $W$ be a Banach space, often $L_2(\Omega)$, such that $\calh(\Omega)\subset W$, and $\mathcal{F}$ a closed subspace of $\calh(\Omega)$.
On an abstract level, our goal is to obtain upper bounds on the approximation error $\lVert v - P v \rVert_W$ for $v \in \mathcal{F}$ that depend on the 
subspace $\mathcal{F}$. Because smoother functions are easier to approximate than rough ones, the magnitude of these bounds should reflect the smoothness of 
$\mathcal{F}$.

This setting covers \emph{kernel interpolation} and \emph{generalized kernel interpolation}, the two most important examples in kernel-based 
approximation~\cite{wendland2005scattered}.
Let $X = \{ x_1, \ldots, x_n \} \subset \Omega$ be a set of points and define $V(X) \coloneqq \operatorname{span}\{ k(\cdot, x_1), \ldots, k(\cdot, x_n)\} 
\subset \calh(\Omega)$.
The kernel interpolant to a function $f\in \calh(\Omega)$ is a function $I_X f \in V(X)$ that interpolates $f$ at $X$, which is to say that $(I_X f)(x_i) = 
f(x_i)$ for every $i = 1, \ldots, n$.
The interpolant can be written as
\begin{equation*}
  I_X f = \sum_{i=1}^n a_i k(\cdot, x_i),
\end{equation*}
where $\mathsf{a} = (a_1, \ldots a_n) = \mathsf{K}_X^\dagger \mathsf{y}$ with $\mathsf{y} = (f(x_1), \ldots, f(x_n))$, and where $(\mathsf{K}_X)^\dagger$ 
is 
the pseudo-inverse of the Gram matrix $\mathsf{K}_X = (k(x_i, x_j))_{i,j=1}^n$. 
If the points are pairwise distinct, then $\mathsf{K}_X$ is typically  
non-singular.
This is always the case if $k$ is strictly positive definite.
If $k$ is strictly positive definite, the kernel interpolant is well-defined even if $f \notin \calh(\Omega)$.
If $f \in \calh(\Omega)$, then the kernel interpolation operator $I_X \colon \calh(\Omega) \to V(X)$ is the $\calh(\Omega)$-orthogonal projection  onto 
$V(X)$.
More generally, one can take any collection $\mathcal{L} = \{ L_i \}_{i=1}^n \subset \calh(\Omega)'$ of continuous linear functionals and define the generalised 
interpolation operator $I_\mathcal{L} \colon \calh(\Omega) \to V(\mathcal{L})$ as the $\calh(\Omega)$-orthogonal projection onto $V(\mathcal{L}) \coloneqq 
\Sp\{k_{L} : L \in \mathcal{L}\}\subset\calh(\Omega)$, the span of their representers.
As their names suggest, the kernel and generalized kernel interpolants satisfy $(I_X f)(x) = f(x)$ for all $x \in X$ and $L(I_\mathcal{L} f) = L(f)$ for all $L \in \mathcal{L}$.

\subsection{Mercer expansions}\label{sec:mercer}

For notational simplicity we assume that $\Omega$ is a subset of $\R^d$.
The Hilbert space of equivalence classes of functions that are square-integrable with respect to the Lebesgue measure on $\Omega$ is denoted $L_2(\Omega)$.
This space is equipped with the standard inner product $\langle f, g \rangle_{L_2(\Omega)} = \int_\Omega f(x) g(x) \mathrm{d} x$.
Throughout the article $\Omega$ and $k$ are assumed such that Mercer's theorem holds.
The precise meaning of this assumption is encapsulated below.

\begin{assumption}[Mercer] \label{ass:mercer}
The set $\Omega$ is a subset of $\R^d$ and $k$ is a positive semi-definite kernel on $\Omega$ such that 
\begin{align}\label{eq:T_operator}
(Tf)(x) \coloneqq \int_\Omega k(x, y) f(y) \mathrm{d} y \quad \text{ for } \quad f\in L_2(\Omega) \: \text{ and } \: x \in \Omega
\end{align}
defines a linear integral operator $T \colon L_2(\Omega)\to L_2(\Omega)$ with a countable collection of $L_2(\Omega)$-orthonormal eigenfunctions 
$\{\varphi_j\}_{j=1}^\infty \subset L_2(\Omega)$ and positive eigenvalues $\{\lambda_j\}_{j=1}^\infty\subset(0,\infty)$ such that $\{ \lambda_j^{1/2} \varphi_j 
\}_{j=1}^\infty$ is an orthonormal basis of $\calh(\Omega)$.
\end{assumption}

\begin{rem} \label{remark:mercer-assumptions}
  We assume that $\Omega \subset \R^d$ and that $L_2(\Omega)$ is with respect to the Lebesgue measure only to simplify notation.
  For example, eigenfunctions and eigenvalues as specified in \Cref{ass:mercer} exist if $\Omega$ is a compact metric space, $k$ is strictly positive definite and continuous on $\Omega \times \Omega$ and integration in~\eqref{eq:T_operator} and $L_2(\Omega)$ are defined with respect to a finite Borel measure $\mu$.
  More general conditions can be found in \cite{steinwart2012mercer}.
  Many of our results in \Cref{sec:superconv_mercer} remain true in this more general setting.
\end{rem}
Suppose that \Cref{ass:mercer} holds. 
Because $\{ \lambda_j^{1/2} \varphi_j \}_{j=1}^\infty$ is an orthonormal basis of $\calh(\Omega)$, it follows (e.g., Section~2.1 in \cite{Paulsen2016}) that the 
kernel has the pointwise convergence Mercer expasion
\begin{equation} \label{eq:mercer}
  k(x, y) = \sum_{j=1}^\infty \lambda_j \varphi_j(x) \varphi_j(y) \quad \text{ for all } \quad x, y \in \Omega .
\end{equation}
Moreover, any $f\in L_2(\Omega)$ can be written as
\begin{equation}\label{eq:L2_expansion}
f=\sum_{j=1}^\infty \langle f, \varphi_j \rangle_{L_2(\Omega)} \varphi_j, 
\end{equation}
and $\norm{L_2(\Omega,\mu)}{f}^2 = \sum_{j=1}^\infty |\langle f, \varphi_j \rangle_{L_2(\Omega)}|^2$.
In particular, the operator $T$ in~\eqref{eq:T_operator} can be represented as 
\begin{equation}\label{eq:T_as_series}
Tf = \sum_{j=1}^\infty \lambda_j\inner{L_2(\Omega)}{f, \varphi_j}\varphi_j \quad \text{ for } \quad f\in L_2(\Omega).
\end{equation}
Because $\{ \sqrt{\lambda_j} \varphi_j \}_{j=1}^\infty$ is an orthonormal basis of $\calh(\Omega)$,
\begin{align*}
\calh(\Omega)=\left\{f \coloneqq \sum_{j=1}^\infty \langle f, \varphi_j \rangle_{L_2(\Omega)} \varphi_j: \sum_{j=1}^\infty \frac{|\langle f, \varphi_j \rangle_{L_2(\Omega)}|^2}{\lambda_j} < \infty \right\}\subset L_2(\Omega).
\end{align*}

\subsection{Sobolev spaces and kernels}\label{sec:sobolev_kernels}
As a relevant special case to be discussed in detail, we consider kernels whose RKHSs are norm equivalent to Sobolev spaces on suitable sets 
$\Omega\subset\R^d$.
Let $\tau > 0$ and define the Bessel potential space 
\begin{equation*}
  H^\tau(\R^d) = \bigg\{ f \in L_2(\R^d) : \lVert f \rVert_{H^\tau(\R^d)}^2 \coloneqq \int_{\R^d} (1 + \lVert \omega \rVert_2^2)^\tau \lvert \hat{f}(\omega) 
\rvert^2 \mathrm{d} \omega < \infty \bigg\} .
\end{equation*}
Here $\hat{f}$ denotes the Fourier transform of $f$.
The restriction of $H^\tau(\R^d)$ on an arbitrary $\Omega \subset \R^d$ is
\begin{align} 
\begin{aligned}
\label{eq:Bessel-Omega}
H^\tau(\Omega) &= \{ f \in L_2(\Omega) : f = f_e|_\Omega \text{ for } f_e \in H^\tau(\R^d) \}, \\
\Vert f \Vert_{H^\tau(\Omega)} &= \inf \{ \Vert f_e \Vert_{H^\tau(\R^d)} : f = f_e|_\Omega \text{ for } f_e \in H^\tau(\R^d) \} .
\end{aligned}
\end{align}
We say that a positive-definite kernel $k$ on $\Omega \subset \R^d$ is a \emph{Sobolev kernel} of order $\tau > d/2$ if its RKHS, $\calh(\Omega)$, is 
norm equivalent to $H^\tau(\Omega)$.
This norm equivalence is denoted $\calh(\Omega) \asymp H^\tau(\Omega)$.
The elements of $H^\tau(\Omega)$ are in general equivalence classes rather than functions but an RKHS consists of functions.
The condition $\tau > d/2$ is needed because then the Sobolev embedding theorem ensures that the elements of $H^\tau(\Omega)$ are continuous functions.
If $\Omega$ is sufficiently regular, such as an open and bounded set with Lipschitz boundary, then 
$H^\tau(\Omega)$ is norm equivalent to the fractional Sobolev space $W_2^\tau(\Omega)$ defined in the standard way via weak derivatives or a H\"older-type 
condition.
If $\tau=k\in\N$, for $1\leq p< \infty$ these spaces are equipped with the seminorm and norm
\begin{equation*} 
\seminorm{W_{p}^k(\Omega)}{u}=\bigg(\sum_{|\alpha|=k}\lVert D^{\alpha}g\lVert^p_{L_p(\Omega)}\bigg)^{1/p},\quad 
\norm{W_{p}^k(\Omega)}{u}=\bigg(\sum_{|\alpha|\le k}\lVert D^{\alpha}u\lVert^p_{L_p(\Omega)}\bigg)^{1/p},
\end{equation*}
with the usual extension for $p=\infty$.
Here $D^\alpha$ is the partial derivative with multiindex $\alpha\in\N_0^d$.
In the general case when $\tau=k+s$ with $k\in\N$ and $0<s< 1$,
we set instead
\begin{equation*}
\seminorm{W_p^{k+s}(\Omega)}{u}
\coloneqq \left(\sum_{|\alpha|=k} \int_{\Omega\times\Omega} \frac{|D^\alpha u(x) -D^\alpha u(y)|^p }{\norm{2}{x-y}^{d+ps}} \mathrm{d}x \mathrm{d}y \right)^{1/p}
\end{equation*}
and
\begin{equation*}\label{eq:sobolev_norm}
\norm{W_{p}^\tau(\Omega)}{u}\coloneqq\left(\norm{W_p^k(\Omega)}{u}^p +\seminorm{W_p^{k+s}(\Omega)}{u}^p\right)^{1/p}.
\end{equation*}

Translation-invariant kernels constitute a popular class of kernels.
A kernel $k$ on $\R^d$ is translation-invariant if there is $\Phi \colon \R^d \to \R$ such that $k(x, y) = \Phi(x - y)$ for all $x, y \in \R^d$.
If $\Phi$ is continuous and integrable, then $k$ is a Sobolev kernel of order $\tau > d/2$ if there exist $c_\Phi, C_\Phi > 0$ such that
\begin{align}
\label{eq:fourier_decay}
c_\Phi (1 + \|\omega\|_2^2)^{-\tau}\leq \hat\Phi(\omega)\leq C_\Phi (1 + \|\omega\|_2^2)^{-\tau}\;\;\fa \;\;\omega\in\R^d.
\end{align}
This implies that $\calh(\R^d)$ is norm equivalent to the fractional Sobolev space $W_2^\tau(\R^d)$, 
and the result 
holds also for restrictions if $\Omega$ is sufficiently regular.
For example, $\calh(\Omega)$ is norm equivalent to $W_2^\tau(\Omega)$ if $\tau > d/2$ and $\Omega$ is an open and bounded set with Lipschitz 
boundary~\cite[Corollary~10.48]{wendland2005scattered}.
The translation-invariant Matérn kernel, which is given by
\begin{equation*}
  \Phi(z) = \frac{2^{1-\nu}}{\Gamma(\nu)} \big( \sqrt{2\nu} \varepsilon \lVert z \rVert_2 \big)^\nu \mathrm{K}_\nu( \sqrt{2\nu} \varepsilon \lVert z \rVert_2)
\end{equation*}
for positive parameters $\nu$ and $\varepsilon$, is an example of a Sobolev kernel of order $\tau = \nu + d/2$.
Here $\Gamma$ is the gamma function and $\mathrm{K}_\nu$ the modified Bessel function of the second kind of order $\nu$.
However, not all Sobolev kernels are translation-invariant.
For example, the released Brownian motion kernel $k(x, y) = 1 + \min\{x, y\}$ is a Sobolev kernel of order $\tau = 1$ on any bounded interval $\Omega \subset 
[0, \infty)$.

The behaviour of the eigenvalues $\lambda_j$ of the integral operator $T$ in~\eqref{eq:T_operator} is well understood if the kernel is Sobolev.
Namely, if $\Omega \subset \R^d$ is an open and bounded set with a sufficiently regular boundary and $k$ a Sobolev kernel of order $\tau$, then there are 
constants $c, C > 0$ such that 
\begin{equation}\label{eq:eigen_decay_sobolev}
  c j^{-2\tau/d} \leq \lambda_j \leq C j^{-2\tau/d}
\end{equation}
for all $j \in \N$~\cite[p.\@~370]{Steinwart2019}.
Moreover, for these kernels there are well-known error bounds on the interpolation error. They depend in particular on the distribution of the points 
$X\subset\Omega$ as measured by the fill distance
\begin{equation}\label{eq:fill_dist}
h_X \coloneqq h_{X, \Omega} \coloneqq \sup_{x \in \Omega} \min_{x_j \in X} \Vert x - x_j \Vert_2.
\end{equation}
We have the following fundamental result for kernel interpolation (e.g., Corollary 11.33 in~\cite{wendland2005scattered}).

\begin{theorem}\label{th:std_error_bound}
Let $\Omega\subset\R^d$ be a bounded set with a Lipschitz boundary that satisfies an interior cone condition. Let $q\in[1,\infty]$, $m\in\N_0$, and let $k$ be a Sobolev kernel of order $\tau>m+d/2$.
Then there are constants $C, h_0 >0$ such that for all discrete sets $X\subset\Omega$ with $h_X\leq h_0$ it holds
\begin{equation}\label{eq:std_error_bound}
\seminorm{W_q^m(\Omega)}{f - I_X f} \leq C h_X^{\tau - m - d(1/2 - 1/q)_+} \norm{\calh(\Omega)}{f-I_X f},\;\; f\in \calh(\Omega),
\end{equation}
where $(x)_+\coloneqq \max(0,x)$.
\end{theorem}

\subsection{Superconvergence}
\label{subsec:background_superconv}

The usual superconvergence results \cite{schaback2018superconvergence} apply to 
the image of $L_2(\Omega)$ under the integral operator $T$ in~\eqref{eq:T_operator}, which is
\begin{equation*}
T(L_2(\Omega))
= \{ Tf : f \in L_2(\Omega)\} 
=\left\{f \coloneqq \sum_{j=1}^\infty \langle f, \varphi_j \rangle_{L_2(\Omega)} \varphi_j: \sum_{j=1}^\infty \frac{|\langle f, \varphi_j \rangle_{L_2(\Omega)}|^2}{\lambda_j^2} < \infty \right\}.
\end{equation*}
These results allow one to use estimates on the $L_2(\Omega)$-error of a projection in the entire RKHS to estimate the $\calh(\Omega)$-error in $TL_2(\Omega)$.
The following superconvergence theorem is a special case of Theorem~1 in~\cite{sloan2025doubling} and contained in the proof of Theorem~11.23 
in~\cite{wendland2005scattered}.

\begin{theorem} \label{thm:usual-superconvergence}
If
\begin{equation}\label{eq:std_sch}
  \lVert f - P f \rVert_{L_2(\Omega)} \leq \varepsilon \lVert f \rVert_{\calh(\Omega)} \;\; \fa \;\; f \in \calh(\Omega),
\end{equation}
then 
\begin{equation}\label{eq:super_sch_h}
  \lVert v - P v \rVert_{\calh(\Omega)} \leq \varepsilon \lVert v \rVert_{L_2(\Omega)} \;\; \fa \;\; v \in T(L_2(\Omega)).
\end{equation}
\end{theorem}
A bound like~\eqref{eq:super_sch_h} on the $\calh(\Omega)$-norm of the error leads superconvergence of double order for $v\in (T 
L_2(\Omega))$: Taking $f\coloneqq v-Pv$ leads to $P f =0$, and thus~\eqref{eq:std_sch} turns into
\begin{equation*}
\lVert v - P v \rVert_{L_2(\Omega)} \leq \varepsilon^2 \lVert v \rVert_{L_2(\Omega)}.
\end{equation*}
For Sobolev kernels this bound and \Cref{th:std_error_bound} yield the following well-known corollary (see~\cite{hangelbroek2025extending} for similar bounds for stronger Sobolev norms).

\begin{cor} \label{cor:standard-sobolev-superconvergence}
  Consider the setting of \Cref{th:std_error_bound} and let $C_{m,q}, h_{0,m,q} > 0$ be the constants in this theorem with their dependency on $m$ and $q$ made 
explicit.
  Then for all discrete sets $X\subset\Omega$ with $h_X\leq \min\{h_{0,m,q}, h_{0,0,2}\}$ it holds
  \begin{equation*}
  \seminorm{W_q^m(\Omega)}{v - I_X v} \leq C_{m,q} C_{0,2} h_X^{2\tau - m - d(1/2-1/q)_+} \norm{L_2(\Omega)}{v},\;\; v \in T (L_2(\Omega)) .
  \end{equation*}
\end{cor}

We furthermore remark that the proof of Theorem~\ref{thm:usual-superconvergence} leverages \smash{$\lVert f - Pf \rVert_{\calh(\Omega)}^2 = 
\langle f - Pf, f \rangle_{\calh(\Omega)}$} 
and 
the identity (10.8) in~\cite{wendland2005scattered}, i.e., 
\begin{align}
\label{eq:from_wendland_book}
\langle f, v \rangle_{L_2(\Omega)} = \langle f, Tv \rangle_{\calh(\Omega)} \;\fa f\in\calh(\Omega), v\in L_2(\Omega).
\end{align}
Moreover, the space $T(L_2(\Omega))$ can also be interpreted as the RKHS of the convolutional kernel $k*k$, that is defined as
\begin{align}
(k*k)(x, z) = \int_{\Omega} k(x, y)k(y, z) \mathrm{d}y.
\end{align}

\section{General superconvergence}\label{sec:cont_superconv}

We start by formulating a general result on superconvergence, which generalizes those of~\cite{schaback2018superconvergence,sloan2025doubling} that double the  
order of convergence for certain classes of functions. 

In this section all results hold for a general Hilbert space $\calh$, even if our focus in the rest of this paper will be on the RKHS $\calh(\Omega)$ 
of 
a given kernel.
We remark additionally that in the statement there is some freedom to interchange the roles of $V, V', A, A^*$, since $(A^*)^*=A$ and $V''=V$ (if $V$ is 
reflexive). Different choices lead to possibly simpler notations in the concrete examples that will follow in the next sections.
Here and in the following, we denote by $\call(V, W)$ the space of bounded linear functionals between Banach spaces $V$ and $W$.

Observe that the following theorem generalizes Theorem 1 in~\cite{schaback2018superconvergence}, 
where a general linear operator and its adjoint are considered in place of an embedding operator.

\begin{theorem}\label{th:super_A_star}
Let $\calh$ be a Hilbert space, $V$ be a Banach space, and let $A\in\call(\calh, V)$ be a linear and continuous operator and $A^*\in\call(V', \calh)$ its adjoint operator.

Assume that for an orthogonal projection $P\colon \calh\to\calh$ there exist a constant $\varepsilon>0$ such that for all $f\in \calh$ there 
is an error bound
\begin{equation}\label{eq:worst_case_V_star}
\norm{V}{A(f - P f)} \leq \varepsilon \norm{\calh}{f}.
\end{equation}
Then for $v \in A^*(V')$ there is a refined error bound
\begin{align}\label{eq:superconv_V}
\norm{\calh}{v-P v} \leq \norm{A^*(V')}{v}\cdot \varepsilon,
\end{align}
where
\begin{align}\label{eq:A_norm}
\norm{A^*(V')}{v}\coloneqq \inf\{\norm{V'}{g}: v = A^* g\}.
\end{align}
\end{theorem}
\begin{proof}
Using the fact that $P$ is an orthogonal projection, we have
\begin{align*}
\sup_{\norm{\calh}{f}\leq 1}  \left|\inner{\calh}{v, f - P f}\right|
=\sup_{\norm{\calh}{f}\leq 1}  \left|\inner{\calh}{v-P v, f - P f}\right|
&\leq \sup_{\norm{\calh}{f}\leq 1}  \norm{\calh}{v-P v}\norm{\calh}{f - P f}\\
&\leq \sup_{\norm{\calh}{f}\leq 1}  \norm{\calh}{v-P v}\norm{\calh}{f}\\
&= \norm{\calh}{v-P v},
\end{align*}
where both inequalities are equalities if
\begin{equation*}
f \coloneqq \left(v-P v\right)/ \norm{\calh}{v-P v}.
\end{equation*}
Thus
\begin{align} \label{eq:proj_norm_as_sup}
\norm{\calh}{v-P v}=\sup_{\norm{\calh}{f}\leq 1}  \left|\inner{\calh}{v, f - P f}\right|\;\;\fa v\in\calh.
\end{align}
By definition of $A$, we now have
\begin{align*}
\inner{\calh}{A^* g, f} = g(A f) \;\;\fa f\in \calh, g\in V',
\end{align*}
and in particular for $v= A^* g$ with $g\in V'$ we can write
\begin{align*}
\left|\inner{\calh}{v, f}\right|
= \left|\inner{\calh}{A^* g, f}\right|
= \left|g(A f)\right|
\leq \norm{V'}{g} \norm{V}{A f}\;\;\fa f\in\calh.
\end{align*}
Combining \eqref{eq:proj_norm_as_sup} and the previous inequality we get
\begin{equation*}
\norm{\calh}{v-P v}
= \sup_{\norm{\calh}{f}\leq 1}  \left|\inner{\calh}{v, f - P f}\right|
\leq \norm{V'}{g}\cdot \sup_{\norm{\calh}{f}\leq 1} \norm{V}{A (f - Pf)},
\end{equation*}
and using~\eqref{eq:worst_case_V_star} completes the proof since $g$ such that $v=Ag$ is arbitrary.
\end{proof}

\begin{rem}
The theorem proves superconvergence for functions that are in the range of a certain operator, and we will see in the following that, in some cases of interest, 
this range can be characterized quite explicitly in terms of additional conditions that a function needs to satisfy. 
Although we do not address this issue in this paper, it would be of interest to characterize an optimal operator, in the sense of guaranteeing maximal 
superconvergence rates while requiring minimal assumptions.
\end{rem}

Theorem~\ref{th:super_A_star} can be proven also in the following modified setting.

\begin{theorem}\label{th:super_direct} 
Let $V, A$ be as in Theorem~\ref{th:super_A_star}, but instead of assuming~\eqref{eq:worst_case_V_star}, assume that there is $v \in A^*(V')$ such that for an orthogonal projection $P\colon \calh\to\calh$ there is a constant 
$\varepsilon>0$ such that 
\begin{equation}\label{eq:worst_case_V_direct}
\norm{V}{A(v - P v)} \leq \varepsilon\norm{\calh}{v-Pv}.
\end{equation}
Then \eqref{eq:superconv_V} holds.
\end{theorem}
\begin{proof}
Let $v = A^* g$ for $g \in V'$.
Using the definition of $A^*$, and the fact that $P$ is an orthogonal projection, we obtain
\begin{align*}
\norm{\calh}{v-P v}^2
=\left|\inner{\calh}{v-Pv, v-Pv}\right|
=\left|\inner{\calh}{v, v-Pv}\right|
= \left|\inner{\calh}{A^*g, v-Pv}\right| &= \left|g(A(v-Pv))\right| \\
&\leq \norm{V'}{g} \norm{V}{A(v-Pv)},
\end{align*}
and using~\eqref{eq:worst_case_V_direct} completes the proof.
\end{proof}

\begin{rem}
\label{rem:condition}
Observe that the condition~\eqref{eq:worst_case_V_direct} is weaker than condition~\eqref{eq:worst_case_V_star}: it is implied by taking $f\coloneqq 
v-Pv$, so that $f-Pf = v-Pv$ by the orthogonality of $P$, and it needs 
to hold only for the single function $v$, rather than for every $f\in \calh$. However, we still need the formulation of \Cref{th:super_A_star} for the results 
of the next section.

The formulation of \Cref{th:super_direct} may be particularly useful in the case of adaptive algorithms such as greedy 
methods, where suitable projections are constructed for single functions.
In particular, when $\calh$ is the RKHS of a strictly positive definite kernel, and for the choice $V' = L_\infty(\Omega)$, we are in the setting of the 
greedy algorithms considered in~\cite{wenzel2023analysis}, while for $V' = W^{\ell}_{\infty}(\Omega)$ we can cover the setting of~\cite{wenzel2025adaptive}.
\end{rem}

\subsection{Intermediate orders}\label{sec:intermediate_orders}
We have seen in Theorem~\ref{th:super_A_star} that superconvergence occurs in the image $A^*(V')$ of operators $A\in\call(\calh, V)$.
We now extend these superconvergence results to a full range of subspaces that are in a certain way intermediate between $A^*(V')$ and $\calh$: These will provide less 
restrictive conditions for superconvergence, at the price of a smaller improvement in the rates.

To this end we first recall some notions from interpolation theory, for which we refer to Chapter~1 in~\cite{lunardi2018interpolation}.
Consider a Banach space $(\mathcal K, \norm{\mathcal K}{})$ that is a subspace of the Hilbert space $\calh$. For $f\in\calh$ and $t>0$, we define the 
$K$-functional
\begin{equation}\label{eq:K_funct}
K(t, f)\coloneqq \inf_{v\in \mathcal K} \left(\norm{\calh}{f-v} + t \norm{\mathcal K}{v}\right),
\end{equation}
and for $p\in[1, \infty]$, $\theta \in (0,1)$ we define the norms
\begin{align}\label{eq:interp_norm}
\norm{\theta,p}{f} &\coloneqq \norm{L_{p,*}(0,\infty)}{t\mapsto t^{-\theta} K(t, f)},\;\;1\leq p<\infty, \\
\norm{\theta,\infty}{f} &\coloneqq \sup_{t>0} t^{-\theta} K(t, f),\nonumber
\end{align}
where $\norm{L_{p,*}(0,\infty)}{f}^p \coloneqq \int_0^\infty |f(t)|^p/t \, \mathrm{d} t$.
The real $p$-interpolation space between $\calh$ and $\mathcal K$ is  then the space
\begin{equation*}
(\calh, \mathcal K)_{\theta,p}\coloneqq \{f\in\calh: \norm{\theta,p}{f}<\infty\},
\end{equation*}
and it is a Banach space with the norm $\norm{(\calh, \mathcal K)_{\theta,p}}{}\coloneqq\norm{\theta,p}{}$. 
Observe moreover that $(\calh, \mathcal K)_{\theta,p} = (\mathcal K, \calh)_{1-\theta,p}$, and it holds (see Proposition 1.1.3 
in~\cite{lunardi2018interpolation}) that
\begin{equation}\label{eq:interp_space_inclusion}
    (\calh, \mathcal K)_{\theta,p} \subset (\calh, \mathcal K)_{\theta,q},\;\; 1\leq p \leq q\leq \infty,\;\theta\in(0,1).
\end{equation}

The following corollary extends Theorem~\ref{th:super_A_star} for interpolation spaces between $\calh$ and $A^*(V')$.
Observe that the condition $v\in (\calh, A^*(V'))_{\theta,\infty}$ in this corollary is quite weak in view of~\eqref{eq:interp_space_inclusion}, since it is implied by $v\in 
(\calh, A^*(V'))_{\theta,p}$ for any $p\in[1,\infty)$.

\begin{cor}\label{cor:intermediate}
Under the assumptions of Theorem~\ref{th:super_A_star}, let $\theta\in(0,1)$ and let $v\in (\calh, A^*(V'))_{\theta,\infty}$. Then there is a refined 
error bound
\begin{equation}\label{eq:intermediate_conv}
\norm{\calh}{v-P v} \leq \norm{(\calh, A^*(V'))_{\theta,\infty}}{g}\cdot \varepsilon^{\theta}.
\end{equation}
\end{cor}
\begin{proof}
The proof follows some ideas used in Section 4 in~\cite{li2024entropy}.
First, consider $f\in\calh$ and any $v_0\in A^*(V')$. By the triangle inequality and properties of the orthogonal projection we have
\begin{equation*}
\norm{\calh}{f-Pf} 
\leq \norm{\calh}{f-v_0}+\norm{\calh}{v_0-Pv_0} + \norm{\calh}{P v_0- Pf}
\leq 2 \norm{\calh}{f-v_0}+\norm{\calh}{v_0-Pv_0},
\end{equation*}
and applying Theorem~\ref{th:super_A_star} to $v_0$ we get $\norm{\calh}{v_0-Pv_0}\leq \varepsilon \norm{A^*(V')}{v_0}$ with the norm defined in~\eqref{eq:A_norm}.
Thus
\begin{align}\label{eq:split}
\norm{\calh}{f-Pf}
\leq 2 \norm{\calh}{f-v_0}+\varepsilon \norm{A^*(V')}{v_0}
\leq 2 \left(\norm{\calh}{f-v_0}+\varepsilon \norm{A^*(V')}{v_0}\right).
\end{align}
Let $\theta\in(0,1)$ and $v\in (\calh, \mathcal K)_{\theta,\infty}$. Since $v_0\in A^*(V')$ is arbitrary in~\eqref{eq:split}, we have that
\begin{align*}
\norm{\calh}{v - Pv} 
\leq 2 \inf_{v_0\in A^*(V')} \left(\norm{\calh}{v-v_0}+\varepsilon \norm{A^*(V')}{v_0}\right)\leq 2 K(\varepsilon, v), 
\end{align*}
where we used~\eqref{eq:K_funct} in the last step. 
Finally, using~\eqref{eq:interp_norm} gives
\begin{align*}
    K(\varepsilon, v) = \varepsilon^\theta \cdot \varepsilon^{-\theta} K(\varepsilon, v) \leq \varepsilon^\theta \sup_{\varepsilon > 0} \varepsilon^{-\theta} 
K(\varepsilon, v) = \varepsilon^\theta \lVert v \rVert_{(\calh, A^*(V'))_{\theta,\infty}},
\end{align*}
which concludes the proof. 

\end{proof}
\begin{rem}
Observe that for $\theta\to1$ we get $v\in A^*(V')$, and the convergence rate~\eqref{eq:intermediate_conv} is that of Theorem~\ref{th:super_A_star}.
For $\theta \to 0$, on the other hand, we are approaching the case $v\in\calh$, for which~\eqref{eq:intermediate_conv} gives the trivial bound 
$\norm{\calh}{v-Pv}\leq \norm{\calh}{v}$.
\end{rem}

In the following we will be mostly concerned with the case when $V=L_p(\Omega)$ for some $p\in[1, \infty)$. 
In particular, we will deal in Section~\ref{sec:superconv_mercer} with the notable case when $A$ is the embedding operator in $V=L_2(\Omega)$ and thus $A^*=T$ 
(see~\eqref{eq:T_operator}), and in Section~\ref{sec:superconv_in_images} we will consider a general $A$ with $V=L_p(\Omega)$ for $p\in[1, \infty)$. Also in 
this case we will prove that $A^*$ is an integral operator. 
However, we stress that Theorem~\ref{th:super_A_star} and Corollary~\ref{cor:intermediate} are not limited to this setting, as shown in the 
following example.

\begin{example}\label{ex:l1}
We start from Example 2.2 in~\cite{santin2021sampling}, that we reformulate with the notation of this paper. 
Let $\calh(\Omega)$ be the RKHS of a positive definite kernel $k$ on a set $\Omega\subset\R^d$. 
Let $I\subset\N$ and let $V\coloneqq \ell_\infty(I)$, so that $V'=\ell_1(I)$. 
Let $Z\coloneqq\{z_i\}_{i\in I}\subset \Omega$. 
Define the operator $A^*\in\call(V^*,\calh(\Omega))$ by 
\begin{equation*}
A^*(\rho)\coloneqq \sum_{i\in I} \rho_i k(\cdot, z_i).
\end{equation*}
Taking $\rho\in \ell_1(I)$ and $f\in\calh(\Omega)$, we have
\begin{equation}\label{eq:example_linf}
\inner{\calh(\Omega)}{A^*(\rho), f}
= \sum_{i\in I} \rho_i f(z_i)
= \rho(f_Z),
\end{equation}
where $f_Z\coloneqq (f(z_i))_{i\in I}$ is a sequence in $\ell_\infty(I)$, since 
\begin{equation*}
\norm{\ell_\infty(I)}{f_Z} = \sup_{i\in I}|f(z_i)|\leq \norm{L_{\infty}(\Omega)}{f}\leq \norm{\calh(\Omega)}{f}<\infty. 
\end{equation*}
Thus $A\in\call(\calh(\Omega), V)$ is defined by $f\mapsto f_Z$.
If $k$ is a Sobolev kernel and $P=I_X$ for a finite set $X\subset\Omega$ with sufficiently small fill distance, then using again~\eqref{eq:example_linf} and 
Theorem~\ref{th:std_error_bound} we obtain the bound~\eqref{eq:worst_case_V_direct} in the form
\begin{align*}
\norm{\ell_\infty(I)}{A(f-I_X f)} = \sup_{i\in I}|(f-I_Xf)(z_i)|\leq  \norm{L_{\infty}(\Omega)}{f-I_Xf} \leq C h_X^{\tau-d/2} \norm{\calh(\Omega)}{f},
\end{align*}
where the value of $\varepsilon>0$ corresponds to a standard interpolation error bound when $P=I_X$ for a finite set $X\subset\Omega$. 
Thus, 
Theorem~\ref{th:super_A_star} implies that there is superconvergence for any $v\in\calh(\Omega)$ that can be written as a kernel expansion
\begin{equation*}
v = \sum_{i\in I} \rho_i k(\cdot, z_i),
\end{equation*}
with arbitrary centers $z_i\in\Omega$ and summable coefficients $\sum_{i\in I}|\rho_i|< \infty$. More precisely, \eqref{eq:superconv_V} gives
\begin{equation*}
\norm{\calh(\Omega)}{v-Pv} \leq  C h_X^{\tau-d/2} \norm{\ell_1(I)}{\rho}.
\end{equation*}
In this case, fractional superconvergence happens in $(\calh(\Omega), A^*(V'))_{\theta,\infty}$.
We are not yet able to provide a more explicit characterization of this 
space, even if we expect it to be related to a weaker summability condition on the coefficients.
\end{example}

\section{Superconvergence in Mercer-based spaces} \label{sec:superconv_mercer}
From now on we work under \Cref{ass:mercer}.
That is, we consider a positive definite kernel $k$ on a set $\Omega \subset \R^d$ and assume that the integral operator given by $(Tf)(x) = \int_\Omega k(x, y) f(y) \mathrm{d} y$ has positive eigenvalues $\{\lambda_j\}_{j=1}^\infty$ and $L_2(\Omega)$-orthonormal eigenfunctions $\{\varphi_j\}_{j=1}^\infty$.
As noted in \Cref{remark:mercer-assumptions}, the setting could be generalized so that $\Omega$ that is equipped with a measure $\mu$ instead of the Lebesgue measure.

We start by considering the case introduced in Section~\ref{sec:mercer}, 
which is comprised in the analysis of Section~\ref{sec:cont_superconv} by taking 
$V=V' \coloneqq L_2(\Omega)$, 
and $A$ as the embedding of $\calh(\Omega)$ into $L_2(\Omega)$.
Thus $A^*$ is the kernel integral operator $T$, see \eqref{eq:T_operator}.
We recall that this is the main setting recalled in \Cref{subsec:background_superconv} from \cite{schaback1999improved,schaback2018superconvergence}.

We will see that this case can be analyzed in the setting of intermediate orders of Section~\ref{sec:intermediate_orders}, where we can additionally provide an 
explicit characterization of the real interpolation spaces $(\calh(\Omega), T (L_2(\Omega)))_{\theta,2}$. We study these spaces in 
Section~\ref{sec:power_spaces}, 
and derive corollaries of 
the results of Section~\ref{sec:cont_superconv} in Section~\ref{subsec:superconv_kernel}.

\subsection{Power spaces and interpolation}\label{sec:power_spaces}
We start by recalling a key result from \cite{steinwart2012mercer}. 
\begin{prop} \label{prop:theta_spaces}
For each $\theta\in [0,\infty)$ the \emph{power space}
\begin{align*} 
\calh_\theta(\Omega) \coloneqq \left\{f \coloneqq \sum_{j=1}^\infty \langle f, \varphi_j \rangle_{L_2(\Omega)} \varphi_j: \sum_{j=1}^\infty \frac{|\langle f, \varphi_j \rangle_{L_2(\Omega)}|^2}{\lambda_j^\theta} < \infty \right\}
\end{align*}
is well defined, and it is an Hilbert space with norm 
\begin{equation}\label{eq:theta_norm}
\norm{\calh_\theta(\Omega)}{f}^2 = \sum_{j=1}^\infty \frac{|\langle f, \varphi_j \rangle_{L_2(\Omega)}|^2}{\lambda_j^\theta}.
\end{equation}
Moreover, there exists a limiting value $\theta_0\in (0, 1)$ such that $\calh_\theta(\Omega)$ is an RKHS for $\theta >\theta_0$, and not an RKHS for $\theta\in 
[0,\theta_0]$.
\end{prop}
\begin{proof}
The proof of these facts can be found in \cite{steinwart2012mercer}. We point in particular to the discussion after Proposition 4.2 in 
\cite{steinwart2012mercer} and 
especially to equation (36) for the characterization of the norm. The existence of a limiting $\theta_0$ can be found in \cite{schaback2002approximation}, 
and is discussed in \cite{steinwart2012mercer} as well.
\end{proof}

By the definition of power spaces and of their norms we readily see that 
$\calh_{\theta_1}(\Omega)\hookrightarrow \calh_{\theta_2}(\Omega)$ if and only if
$\theta_1\geq \theta_2$, and from the discussion in Section~\ref{sec:mercer} it holds that
\begin{equation}\label{eq:power_notable_cases}
\calh_0(\Omega) = L_2(\Omega),\quad
\calh_1(\Omega) = \calh(\Omega),\quad
\calh_2(\Omega) = T(L_2(\Omega)).
\end{equation}
We point to \Cref{fig:visualization} for a visualization of the scale of spaces.
\begin{figure}[t]
\newlength\fwidth
\setlength\fwidth{.4\textwidth}
\begin{center}
\begin{tikzpicture}[>=latex, thick]
\draw[->] (0,0) -- (11cm,0) node [right] {$\R$};

\draw (1.5,-4pt) -- (1.5,4pt) node[below=10pt]{0};
\draw (1.5,5pt) node[above, align=center, fill=lightgray, rounded corners, inner sep=2pt]{$L_2(\Omega)$};

\draw (4,-2pt) -- (4,2pt);
\draw (4, 5pt) node[above, align=center, fill=lightgray, rounded corners, inner sep=2pt]{$\mathcal{H}_{\theta}(\Omega)$} node[below=10pt]{$\theta$};

\draw (8,-2pt) -- (8,2pt);
\draw (8, 5pt) node[above, align=center, fill=lightgray, rounded corners, inner sep=2pt]{$TL_1(\Omega)$} node[below=10pt]{$2-\frac{d}{2\tau}$};

\draw (5.5,-2pt) -- (5.5,2pt);
\draw (5.5, 5pt) node[above, align=center, fill=lightgray, rounded corners, inner sep=2pt]{$\calh(\Omega)$} node[below=10pt]{1};
\draw (9.5,-2pt) -- (9.5,2pt);
\draw (9.5, 5pt) node[above, align=center, fill=lightgray, rounded corners, inner sep=2pt]{$TL_2(\Omega)$} node[below=10pt]{2};
\end{tikzpicture}
\end{center}
\caption{Visualization of the scale of power spaces with the special cases $L_2(\Omega)$ (for $\theta = 0$), 
$\calh(\Omega)$ (for $\theta = 1$) and $TL_2(\Omega)$ (for $\theta = 2$).
In the case of Sobolev kernels, $TL_1(\Omega)$ refers to the power space with index $2 - \frac{d}{2\tau}$, see \Cref{th:embedding_power_spaces}.
This paper considers the superconvergence case of intermediate order, i.e.\ $\theta \in (1, 2]$.}
\label{fig:visualization}
\end{figure}
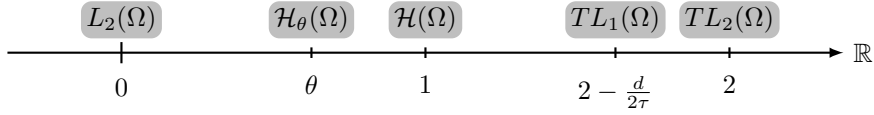
Observe moreover that \eqref{eq:theta_norm} implies that $\norm{\calh_\theta(\Omega)}{\varphi_j}^2 = 1 / \lambda_j^\theta$, 
in accordance with their orthonormality in $L_2$, and with the orthonormality of $\{ \sqrt{\lambda_j} \varphi_j \}_{j=1}^\infty$ in $\calh(\Omega)$ (see 
Section~\ref{sec:mercer}).

Power spaces can be characterized as real interpolation spaces. The proof is an elementary application of interpolation 
theory, that we report for completeness (see also \cite{chandler2015interpolation}).
\begin{lemma}\label{lemma:calh_as_interp}
For $\theta \in (0, 1)\cup(1,2)$ the space $\calh_{\theta}(\Omega)$ is the real $2$-interpolation space between $\calh(\Omega)$ and $L_2(\Omega)$ or $T( 
L_2(\Omega))$, i.e., 
\begin{align*}
\calh_\theta(\Omega) 
= 
\begin{cases}
\left(L_2(\Omega), \calh(\Omega)\right)_{\theta, 2}, & 0 < \theta < 1, \\
\left(\calh(\Omega), T(L_2(\Omega))\right)_{\theta-1, 2}, & 1 < \theta < 2.
\end{cases}
\end{align*}
\end{lemma}
\begin{proof}
Theorem 3.4 in \cite{chandler2015interpolation} proves that if $H_1$ is a Hilbert space that is densely and compactly embedded in a Hilbert space $H_0$, then
the 
operator $S \colon H_1 \rightarrow H_1$ defined by 
\begin{align*}
\inner{H_1}{S f, g} = \inner{H_0}{f, g}\;\;\fa f, g\in H_1
\end{align*}
is compact and self adjoint and it has eigenvalues and eigenfunctions $\mu_j$, $\psi_j$ with $\norm{H_0}{ \psi_j} = 1$, $\{\mu_j\}_{j=1}^\infty$ non increasing and $\mu_j\to 0$. Moreover, for all $\zeta\in (0,1)$ it holds
\begin{align}\label{eq:intermediate_interp_spaces}
\left(H_0, H_1\right)_{\zeta, 2}=\left\{f = \sum_{j=1}^\infty \langle f, \psi_j \rangle_{L_2(\Omega)} \psi_j : \sum_{j=1}^\infty \frac{|\langle f, \psi_j 
\rangle_{L_2(\Omega)}|^2}{\mu_j^\zeta}<\infty \right\},
\end{align}
with squared norm $\norm{\left(H_0, H_1\right)_{\zeta, 2}}{f}^2 = \sum_{j=1}^\infty {|\langle f, \psi_j \rangle_{L_2(\Omega)}|^2}/{\mu_j^\zeta}$.
In our case, we have that $\calh(\Omega)$ is compactly embedded into $L_2(\Omega)$, and the same holds for $T(L_2(\Omega))$ into $\calh(\Omega)$. 
It follows 
that the result can be applied in both cases with $S=T$ (in the second case with the restriction $T\colon \calh(\Omega) \to \calh(\Omega)$), which gives 
$\mu_j=\lambda_j$ in both cases. On the other hand, the normalization $\norm{H_0}{ \psi_j} = 1$ gives $\psi_j=\varphi_j$ in the first case and 
$\psi_j=\varphi_j/\sqrt{\lambda_j}$ in the second one. It follows that the final interpolation spaces on the right of \eqref{eq:intermediate_interp_spaces} are 
exactly $\calh_\theta(\Omega)$ with $\theta\in (0, 1)$ in the first case and $\theta\in (1, 2)$ in the second one.
\end{proof}

\subsection{Fractional superconvergence}\label{subsec:superconv_kernel}
With this characterization of the power spaces we can now obtain a first concrete instance of Corollary~\ref{cor:intermediate}.

\begin{cor}\label{cor:super}
Let $\varepsilon>0$ and let $P\colon \calh(\Omega)\to\calh(\Omega)$ be an orthogonal projection such that
\begin{equation}\label{eq:worst_case_l2}
\norm{L_2(\Omega)}{f-Pf} \leq \varepsilon \norm{\calh(\Omega)}{f}\;\;\fa f\in\calh(\Omega).
\end{equation}
Then for all $\theta\in[0,1]$ and $v\in \calh_{1+\theta}(\Omega)$ it holds
\begin{equation*}
\norm{\calh(\Omega)}{v-P v} 
\leq \varepsilon^{\theta} \norm{\calh_{1+\theta}(\Omega)}{v}.
\end{equation*}
\end{cor}
\begin{proof}
For $\theta=0$ the result is trivial since $\calh_1(\Omega)=\calh(\Omega)$ (see~\eqref{eq:power_notable_cases}).
If $\theta\in (0,1]$, we set $V=V'=L_2(\Omega)$ and let $A\colon \calh(\Omega)\to L_2(\Omega)$ be the embedding operator, so that 
$A^*=T \colon L_2(\Omega)\to\calh(\Omega)$. 
With this setup, the bound~\eqref{eq:worst_case_l2} is an instance of the bound~\eqref{eq:worst_case_V_star}.
For $\theta=1$ we thus obtain the result by simply applying Theorem~\ref{th:super_A_star}, since in this case $\calh_{1+\theta}(\Omega)=\calh_{2}(\Omega) = 
T(L_2(\Omega))$ by~\eqref{eq:power_notable_cases}.

For $\theta\in(0,1)$ we use the fact that $\calh_{1+\theta}(\Omega)=\left(\calh(\Omega), T(L_2(\Omega))\right)_{\theta, 2}\subset \left(\calh(\Omega), 
T(L_2(\Omega))\right)_{\theta, \infty}$ (Lemma~\ref{lemma:calh_as_interp} and equation~\eqref{eq:interp_space_inclusion}). We can thus apply 
Corollary~\ref{cor:intermediate} to $v\in \calh_{1+\theta}(\Omega)$ and get the result. 
\end{proof}

This corollary has an immediate application in the case of Sobolev kernels (see Section~\ref{sec:sobolev_kernels}).

\begin{cor}
\label{cor:application_to_sobolev}
  Consider the setting of \Cref{th:std_error_bound} and let $C_{m,q}, h_{0,m,q} > 0$ be the constants in this theorem with their dependency on $m$ and $q$ made explicit.
  Additional, let $\theta\in[0, 1]$.
  Then for all $v\in \calh_{1+\theta}(\Omega)$ and $X\subset\Omega$ with $h_X\leq h_{0,0,2}$ it holds
\begin{equation*}
\norm{\calh(\Omega)}{v - I_X v} 
\leq C_{0,2}^\theta h_X^{\theta\cdot\tau} \norm{\calh_{1+\theta}(\Omega)}{v} .
\end{equation*}
In particular, for $h_X\leq \min\{ h_{0,m,q}, h_{0,0,2} \}$ it holds
\begin{equation*}
\seminorm{W_q^m(\Omega)}{v - I_X v} 
\leq C_{m,q} C_{0,2}^\theta h_X^{(1+\theta)\tau - m - d(1/2-1/q)_+} \norm{\calh_{1+\theta}(\Omega)}{v}.
\end{equation*}
\end{cor}
\begin{proof}
For a Sobolev kernel and $P=I_X$, we use Theorem~\ref{th:std_error_bound} with $q=2$, $m=0$ to obtain the bound~\eqref{eq:worst_case_l2} with $\varepsilon=C_{0,2} 
h_X^{\tau}$, since $\norm{\calh(\Omega)}{f-I_X f}\leq \norm{\calh(\Omega)}{f}$ for all $f\in\calh(\Omega)$.
We can thus apply Corollary~\ref{cor:super} and insert back the resulting bound in the right-hand side of~\eqref{eq:std_error_bound}, to obtain
\begin{align*}
\seminorm{W_q^m(\Omega)}{v - I_X v} 
&\leq C_{m,q} h_X^{\tau - m - d(1/2-1/q)_+} \norm{\calh(\Omega)}{v-I_X v}\\
&\leq C_{m,q} h_X^{\tau - m - d(1/2-1/q)_+} \left(C_{0,2} h_X^{\tau}\right)^{\theta} \norm{\calh_{1+\theta}(\Omega)}{v}, 
\end{align*}
which is the statement.
\end{proof}

When $\theta \in (0, 1)$, this result provides rates that are between those in \Cref{th:std_error_bound} and \Cref{cor:standard-sobolev-superconvergence}, extending the superconvergence of~\cite{schaback1999improved} (see also Theorem 11.23 in~\cite{wendland2005scattered}) to the entire scale of power 
spaces. 
In the case of Sobolev kernels, as discussed in~\cite{schaback2018superconvergence}, to obtain superconvergence one should not only for additional smoothness but 
also for the satisfaction of some boundary conditions. We will discuss this point in Section~\ref{sec:boundary}. However, there are cases where no boundary 
conditions should be expected and we get an explicit characterization of $\calh_{1+\theta}$ in terms of smoothness only.
In \Cref{ex:zonal}, we consider the sphere which has no boundary.
In \Cref{ex:periodic_kernel}, the kernel instead satisfies strong boundary conditions.

\begin{example}\label{ex:zonal}
This example follows Section 2.2 in~\cite{wendland2020solving}.
We consider the unit sphere $\Omega\coloneqq \mathbb S^{d-1}$ and the corresponding Laplace--Beltrami operator, which has eigenvalues $\mu_j$ and eigenfunctions $\psi_j$, $j\in\N$, giving an orthonormal basis of $L_2(\mathbb S^{d-1})$.
For any $\sigma\geq 0$ and $f\coloneqq\sum_{j=1}^\infty \langle f, \psi_j \rangle_{L_2(\Omega)} \psi_j\in L_2(\mathbb S^{d-1})$ it holds $f\in W_2^\sigma(\mathbb S^{d-1})$ if and only if
\begin{equation}\label{eq:sobolev_sphere}
\sum_{j=1}^\infty |\langle f, \psi_j \rangle_{L_2(\Omega)}|^2 (1 + \mu_j)^\sigma <\infty.
\end{equation}
Moreover, for $\sigma>(d-1)/2$, any kernel on $\mathbb S^{d-1}$ with $\varphi_j=\psi_j$ and
\begin{equation}\label{eq:lambda_vs_mu}
c_1 (1 + \mu_j)^{-\sigma} \leq \lambda_j\leq c_2 (1 + \mu_j)^{-\sigma},
\end{equation}
for some $c_1, c_2>0$, has an RKHS which is norm equivalent to $W_2^\sigma(\mathbb S^{d-1})$.
Any zonal kernel $k:\mathbb S^{d-1}\times\mathbb S^{d-1}\to\R$, i.e., $k(x, y)=\Psi(x^T y)$ for a suitable $\Psi:[-1,1]\to\R$, has the same Mercer 
eigenfunctions of the Laplace--Beltrami operator, and zonal kernels arise in particular as restriction to the sphere of Radial Basis Functions (RBFs) kernels 
defined on $\R^d$.
In particular, the restriction to the sphere of an RBF Sobolev kernel with $\calh(\R^d)\asymp W_2^{\sigma+1/2}(\R^d)$ is a zonal kernel with RKHS 
$W_2^\sigma(\mathbb S^{d-1})$ (Corollary 2.5 in~\cite{wendland2020solving}).
Thus~\eqref{eq:lambda_vs_mu} is satisfied for some $c_1, c_2>0$.
We stress that this construction is not limited to the sphere, but extends also to other manifolds (see Proposition 2.4 in~\cite{wendland2020solving} and the 
following discussion).

Combining these facts, in our setting we have that for $\theta\in [0,1]$ and $f\in L_2(\mathbb S^{d-1})$ it holds
\begin{equation*}
c_1^{1+\theta} \sum_{j=1}^\infty \frac{|\langle f, \psi_j \rangle_{L_2(\Omega)}|^2}{\lambda_j^{1+\theta}}
\leq \sum_{j=1}^\infty |\langle f, \psi_j \rangle_{L_2(\Omega)}|^2 (1 + \mu_j)^{(1+\theta)\sigma}
\leq c_2^{1+\theta} \sum_{j=1}^\infty \frac{|\langle f, \psi_j \rangle_{L_2(\Omega)}|^2}{\lambda_j^{1+\theta}},
\end{equation*}
and using Proposition~\ref{prop:theta_spaces} and the condition~\eqref{eq:sobolev_sphere} this means that $\calh_{1+\theta}(\mathbb 
S^{d-1})=W_2^{(1+\theta)\sigma}(\mathbb S^{d-1})$,
i.e., the power spaces can be solely characterized in terms of smoothness.
\end{example}

\begin{example}
\label{ex:periodic_kernel}
Let $\Omega = [0, 1]$.
The trigonometric polynomials $\varphi_j(x) = e^{2\pi \mathrm{i} j x}$ for $j \in \Z$ are an orthonormal basis of $L_2(\Omega)$ and can be used to define $W_{2,\textup{per}}^\alpha(\Omega)$, the periodic Sobolev space of order $\alpha > 0$ on $[0, 1]$, as
\begin{equation*}
  W_{2,\textup{per}}^\alpha(\Omega) \coloneqq \left\{f \coloneqq \sum_{j \in \Z} \langle f, \varphi_j \rangle_{L_2(\Omega)} \varphi_j \: : \: \sum_{j \in \Z} \lvert j \rvert^{2\alpha} |\langle f, \varphi_j \rangle_{L_2(\Omega)}|^2 < \infty \right\} .
\end{equation*}
Alternatively, one can view $W_{2,\textup{per}}^\alpha(\Omega)$ as $W_2^\alpha(\mathbb{S}^1)$.
When $\alpha$ is an integer this space consists of those functions in the Sobolev space of order $r$ whose derivatives up to order $r-1$ are 
periodic on $[0, 1]$ (i.e., take the same value at $0$ and $1$).
The periodic Sobolev space is an RKHS when $\alpha > 1/2$ and its kernel is
\begin{equation*}
        k_\alpha(x, y) = \sum_{j \in \Z} \lvert j \rvert^{-2 \alpha} e^{2\pi\mathrm{i} j x} \overline{e^{2\pi \mathrm{i} j y}} = 1 + 2 \sum_{j=1}^\infty j^{-2 
\alpha} 
\cos(2\pi j(x-y)) ,
\end{equation*}
where $\overline{z}$ denotes the complex conjugate of $z \in \C$ and we use the convention that $0^{-2\alpha} = 1$.
For an integer order $r \in \N$, the Bernoulli polynomial $\mathrm{B}_{2r}$ of degree $2r$ can be used to write the kernel as 
\begin{equation} \label{eq:periodic-kernel}
  k_r(x, y) = 1 + (-1)^{r+1} (2\pi)^{2r} \frac{\mathrm{B}_{2r}( \lvert x - y \rvert )}{(2r)!} .
\end{equation}
See \cite[Appendix~A.1]{NovakWozniakowski2008} for these results.
Because the functions $\varphi_j(x) = e^{2\pi \mathrm{i} jx}$ for $j \in \Z$ form an orthonormal basis of $L_2(\Omega)$, they are the eigenfunctions of the integral operator $T$ in~\eqref{eq:T_operator} with corresponding eigenvalues $\lambda_j = \lvert j \rvert^{-2\alpha}$.
For any $\theta \in [0, \infty)$ the space $\calh_\theta(k_\alpha, \Omega)$ from Proposition~\ref{prop:theta_spaces} is now
\begin{equation*} 
  \calh_\theta(k_\alpha, \Omega) = \left\{f \coloneqq \sum_{j \in \Z} \langle f, \varphi_j \rangle_{L_2(\Omega)} \varphi_j \: : \: \sum_{j \in \Z} \lvert j \rvert^{2\alpha\theta} |\langle f, \varphi_j \rangle_{L_2(\Omega)}|^2 < \infty \right\}
  = \calh(k_{\alpha \theta}, \Omega) = W_{2,\textup{per}}^{\alpha\theta}(\Omega).
\end{equation*}
\end{example}

\section{Superconvergence in the image of general kernel integral operators}\label{sec:superconv_in_images}
In thi section we need the injectivity of $T$, which is guaranteed since we are working under \Cref{ass:mercer}. 
For more general conditions that ensure this injectivity, we refer to~\cite{sriperumbudur2011univ}

\Cref{sec:superconv_mercer} applied the general result of \Cref{th:super_A_star} to the scale of Mercer-based spaces by choosing $V = L_2(\Omega)$ and $A$ as 
the embedding into $V$.
In this section, we apply the general result of \Cref{th:super_A_star} by choosing $V = L_p(\Omega)$, $1 < p < \infty$ and using any $A \in \mathcal{L}(\calh, 
V)$,
i.e., $A$ is not necessarily an embedding.
In this setting we will observe that the corresponding operator $A^*$ is naturally given as a kernel integral operator, which is the content of 
\Cref{subsec:general_kernel_integral}.
Subsequently, a relation between the images of these kernel integral operators and the Mercer based spaces from \Cref{sec:superconv_mercer} is established in 
\Cref{subsec:characterization_image}.

\subsection{General kernel integral operators}
\label{subsec:general_kernel_integral}

We begin with a result that generalizes the identity~\eqref{eq:from_wendland_book}.

\begin{prop}
\label{prop:A_as_integral_operator}
Assume $A \in \mathcal{L}(\calh(\Omega), L_q(\Omega))$ for some $q \in [1, \infty]$
and let $p \in [1, \infty]$ be the Hölder conjugate of $q$.
Then the adjoint operator $A^* \colon L_p(\Omega) \rightarrow \calh(\Omega)$ is given by the integral operator
\begin{align}
\label{eq:A_as_integral_operator}
(A^*w)(x) \coloneqq \int_\Omega A_y(k(y, x)) w(y) \mathrm{d}y \quad \text{ for all } \quad w \in L_p(\Omega),
\end{align}
i.e., it holds
\begin{align}
\label{eq:dual_pairing_integral}
\langle A f, w \rangle_{L_2(\Omega)} = \langle f, A^* w \rangle_{\calh(\Omega)} \quad \text{ for all } \quad f \in \calh (\Omega), \: w \in L_p(\Omega) .
\end{align}
\end{prop}

We emphasize that we use $\langle A f, w \rangle_{L_2(\Omega)}$ here as a notation referring to the duality pairing $\int_{\Omega} (A f)(x) w(x) \mathrm{d}x$:
In fact, $A f \in L_q(\Omega)$ is not necessarily included in $L_2(\Omega)$, 
however the integral is nevertheless well defined due to 
$|\langle A f, w \rangle_{L_2(\Omega)}| \leq \Vert A f \Vert_{L_q(\Omega)} \cdot \Vert w \Vert_{L_p(\Omega)}$.

\begin{proof}
We start by proving \eqref{eq:dual_pairing_integral} for $f \in \Sp \{ k(\cdot, x) : x \in \Omega\}$.
For $f = \sum_{j=1}^n \alpha_j k(\cdot, x_j)$ we have
\begin{align*}
\langle A f, w \rangle_{L_2(\Omega)} &= \sum_{j=1}^n \alpha_j \langle A k(\cdot, x_j), w \rangle_{L_2(\Omega)} 
= \sum_{j=1}^n \alpha_j \int_{\Omega} A_y (k(y, x_j)) w(y) \mathrm{d}y \\
&= \sum_{j=1}^n \alpha_j (A^* w)(x_j) = \sum_{j=1}^n \alpha_j \langle A^* w, k(\cdot, x_j) \rangle_{\calh(\Omega)} \\
&= \langle f, A^* w \rangle_{\calh(\Omega)}.
\end{align*}
Since $\Sp \{ k(\cdot, x) : x \in \Omega\}$ is dense in $\calh(\Omega)$ by definition and $A \in \mathcal{L}(\calh(\Omega), L_q(\Omega))$, the statement 
follows.
\end{proof}

Note that \eqref{eq:dual_pairing_integral} generalizes \eqref{eq:from_wendland_book} in two directions:
First, it extends the relation from $L_2(\Omega)$ to $L_p(\Omega)$ for $p \in [1, \infty]$ and second, the operator $A$ can be more general than just an 
embedding.
In fact, one could even further generalize \Cref{prop:A_as_integral_operator} by considering an operator $A \in \mathcal{L}(\calh(\Omega), W_p^\tau(\Omega))$ and thus obtaining $A^*$ via the $W_2^\tau(\Omega)$ inner product instead of the $L_2$ inner product,
i.e.,
\begin{align}
\label{eq:A_as_integral_operator_Sobolev}
(A^*w)(x) = \langle Ak(\cdot, x), w \rangle_{W_2^\tau(\Omega)}.
\end{align}

We exemplify the previous \Cref{prop:A_as_integral_operator} with two special cases.

\begin{example}
Let $q  \in [1, \infty]$ and let $A \in \mathcal{L}(\calh(\Omega), L_q(\Omega))$ be the embedding $\calh(\Omega) \hookrightarrow L_q(\Omega)$.
In this case, the operator $A^*$ is the operator $T$ from \eqref{eq:T_operator} and relation \eqref{eq:dual_pairing_integral} is a generalization of \eqref{eq:from_wendland_book} from $L_2(\Omega)$ to $L_p(\Omega)$:
\begin{align}
\label{eq:generalization_Lp}
\langle f, w \rangle_{L_2(\Omega)} = \langle f, Tw \rangle_{\calh(\Omega)} \quad \text{ for all } \quad f \in \calh(\Omega), \: w \in L_p(\Omega).
\end{align}
\end{example}

\begin{example}
Let $q  \in [1, \infty]$ and let $\calh(\Omega)$ be smooth enough such that $A \coloneqq -\Delta \in \mathcal{L}(\calh(\Omega), L_q(\Omega))$.
In this case, the operator $A^*$ is the kernel integral operator of \eqref{eq:A_as_integral_operator} given as
\begin{align*}
(A^*w)(x) = \int_\Omega -\Delta_y k(y, x) w(y) \mathrm{d}y.
\end{align*}
The same works for any other differential operator $A$ as long as $A \in \mathcal{L}(\calh(\Omega), L_q(\Omega))$.
\end{example}

In some cases the range of the adjoint operator $A^*$ can be connected to the range of the operator $T$ in~\eqref{eq:T_operator}.

\begin{example}
Consider the setting of \Cref{ex:periodic_kernel}.
Let $A f = f^{(\beta)}$ be a differentiation operator of order $\beta \in \N_0$ such that all functions in $\calh(k_\alpha, \Omega)$ are $\beta$ times 
differentiable.
For even $\beta = 2 \gamma$ we get 
\begin{equation*}
  A_y (k_\alpha(y, x)) 
  = \sum_{j \in \Z} \lvert j \rvert^{-2 \alpha} (-2\pi j)^\beta e^{2\pi\mathrm{i} j x} \overline{e^{2\pi \mathrm{i} j y}} 
  = (2\pi)^{2\gamma} \sum_{j \in \Z} \lvert j \rvert^{-2(\alpha-\gamma)} e^{2\pi\mathrm{i} j x} \overline{e^{2\pi \mathrm{i} j y}} = (2 \pi)^{2\gamma} k_{\alpha-\gamma}(x, y),
\end{equation*}
so that $A^* = A^*_\alpha$ for $k_\alpha$ equals the operator $T = T_{\alpha-\gamma}$ in~\eqref{eq:T_operator} for the kernel $k_{\alpha-\gamma}$.
We ignore the constant $(2\pi)^{2\gamma}$ as it only scales the RKHS norm.
From the general result $\calh_2(\Omega) = T (L_2(\Omega))$ in Section~\ref{sec:superconv_mercer} and the characterizations of $W_{2,\textup{per}}^\alpha(\Omega)$ and $\calh_\theta(k_\alpha, \Omega)$ it now follows that 
\begin{equation*}
  A^*_\alpha (L_2(\Omega)) = T_{\alpha-\gamma} (L_2(\Omega)) = \calh(k_{\alpha-\gamma}, \Omega) = W_{2,\textup{per}}^{\alpha-\gamma}(\Omega).
\end{equation*}
\end{example}

In this framework,
the general result of \Cref{th:super_A_star} reads as follows.

\begin{cor}
\label{cor:integral_op}
Let $\calh(\Omega)$ be an RKHS with kernel $k$ and choose $V = L_q(\Omega)$ for $1 < q < \infty$.
Let $A \colon \calh \to V$ be a linear and continuous operator and $A^*$ the kernel integral operator given in \eqref{eq:A_as_integral_operator}.

Assume that for an orthogonal projection $P\colon\calh\to\calh$ there exist a constant $\varepsilon>0$ such that for all $f\in \calh$ there 
is an error bound
\begin{equation}\label{eq:worst_case_V_star2}
\norm{V}{A(f - P f)} \leq \varepsilon \norm{\calh}{f}.
\end{equation}
Then for $v\in A^*(V')$ there is a refined error bound
\begin{equation}\label{eq:superconv_V2}
\norm{\calh}{v-P v} \leq \norm{A^*(V')}{v}\cdot \varepsilon,
\end{equation}
where $\norm{A^*(V')}{v}$ is defined as in~\eqref{eq:A_norm}.
\end{cor}

\begin{proof}
This is a combination of \Cref{prop:A_as_integral_operator} and \Cref{th:super_A_star}.
\end{proof}

In order to show that the assumption \eqref{eq:worst_case_V_star2} is easily satisfied in this framework,
we present an application for Sobolev kernels.

\begin{cor} \label{cor:TL_superconvergence-Sobolev}
  Consider the setting of \Cref{th:std_error_bound} and let $C_{m,q}, h_{0,m,q} > 0$ be the constants in this theorem with their dependency on $m$ and $q$ made explicit.
  Let $p, q \in [1, \infty]$ be Hölder conjugates.
  Let $A \colon \calh(\Omega)\to L_q(\Omega)$ with $A f = \mathrm{D}^\beta f = \partial_1^{\beta_1} \cdots \partial_d^{\beta_d} f$ be a differentiation operator for $\beta \in \N_0^d$ such that $m = \lvert \beta \rvert < \tau - d/2$. 
  Assume that $v = A^* g \in A^* (L_p(\Omega))$ with $A^*$ given in \eqref{eq:A_as_integral_operator}.
  Then for all discrete sets $X\subset\Omega$ with $h_X\leq h_{0,m,q}$ it holds
\begin{align} \label{eq:TL_superconvergence-Sobolev-1}
  \lVert v - I_X v \rVert_{\calh(\Omega)} \leq C_{m,q} h_X^{\tau - m - d(1/p -1/2)_+} \lVert g \rVert_{L_p(\Omega)}.
\end{align}
In particular for $m=0$, i.e.~$A_y(k(\cdot, x)) = k(x, y)$, and $h_X \leq h_{0,0}$ we have 
\begin{align*}
  \lVert v - I_X v \rVert_{\calh(\Omega)} \leq C_{0,q} h_X^{\tau - d(1/p -1/2)_+} \lVert g \rVert_{L_p(\Omega)} .
\end{align*}
Moreover, for $t \in \N_0$ such that $t < \tau - d/2$ and $h_X \leq \min\{h_{0,m, q}, h_{0,t,q}\}$ we have
\begin{align} \label{eq:TL_superconvergence-Sobolev-3}
    \seminorm{W_q^t(\Omega)}{v - I_X v} \leq C_{m,q} C_{t,q} h_X^{2\tau - t - m - 2d(1/2 -1/q)_+} \norm{L_p(\Omega)}{g} .
\end{align}

\end{cor}
\begin{proof}
From \Cref{th:std_error_bound} and the fact that the interpolation operator $I_X$ is an orthogonal projection it follows that
\begin{align*}
    \lVert \mathrm{D}^\beta (f - I_X f) \rVert_{L_q(\Omega)} \leq \seminorm{W_q^m(\Omega)}{f - I_X f} &\leq C_{m,q} h_X^{\tau - m - d(1/2 - 1/q)_+} \lVert f - I_X f \rVert_{\calh(\Omega)} \\
    &\leq C_{m,q} h_X^{\tau - m - d(1/2 - 1/q)_+} \lVert f \rVert_{\calh(\Omega)}
\end{align*}
for all $f \in \calh(\Omega)$ when $h_X \leq h_{0,m,q}$.
Since $q$ is the Hölder conjugate of $p$, $1/2 - 1/q = 1/p - 1/2$.
Since $I_X$ is an orthogonal projcetion, the inequality~\eqref{eq:TL_superconvergence-Sobolev-1} then follows from \Cref{cor:integral_op}. Inequality~\eqref{eq:TL_superconvergence-Sobolev-3} follows from \Cref{th:std_error_bound} with $m = t$ and~\eqref{eq:TL_superconvergence-Sobolev-1}.
\end{proof}

Note that within \Cref{cor:TL_superconvergence-Sobolev},
instead of $A \colon \calh(\Omega) \rightarrow L_q(\Omega), f \mapsto D^\beta f$ one could also use $\Vert D^\beta f \Vert_{L_q(\Omega)} \leq | f |_{W_q^{|\beta|}(\Omega)}$ 
and consider $A$ as the embedding $\calh(\Omega) \hookrightarrow W_q^{|\beta|}(\Omega)$ and thus $A^*$ given via \eqref{eq:A_as_integral_operator_Sobolev}. 
This yields essentially the same result.

\subsection{Characterization of images}
\label{subsec:characterization_image}
In the case of Sobolev kernels of smoothness $\tau > d/2$ and $A \colon \calh \to L_p(\Omega)$ being the embedding operator,
the adjoint operator $A^*$ is given as the integral operator $T$ from \eqref{eq:T_operator}.
In this case, 
the spaces $T (L_p(\Omega)) \supset T (L_2(\Omega))$ for $1 \leq p < 2$ describe dense subspaces of Mercer based power spaces.
This is formalized in the following theorem.

\begin{theorem}
\label{th:embedding_power_spaces}
Assume $\Omega \subset \R^d$ bounded, $k$ a Sobolev kernel of smoothness $\tau > d/2$ satisfying
$\sup_{x \in \Omega} k(x, x) < \infty$ and let $T$ be the integral operator from \eqref{eq:T_operator}, which is injective.
Let $T(L_p(\Omega)) = \{ Tv : v \in L_p(\Omega)\}$ be equipped with the norm $\Vert Tv \Vert_{T(L_p(\Omega))} \coloneqq \Vert v \Vert_{L_p(\Omega)}$.
Then it holds for $1 \leq p \leq 2$ and any $\theta < 2 - \frac{2-p}{p} \cdot \frac{d}{2\tau}$:
\begin{enumerate}[label=(\roman*)]
\item $T(L_p(\Omega))$ is a Banach space with continuous point evaluation.
\item $T(L_p(\Omega)) \hookrightarrow \mathcal{H}_\theta(\Omega)$, i.e.\ $T(L_p(\Omega))$ is continuously embedded into $\mathcal{H}_\theta(\Omega)$.
\item $T(L_p(\Omega)) \hookrightarrow \mathcal{H}_\theta(\Omega)$ is dense but not surjective.
\item $\overline{T(L_p(\Omega))}^{\Vert \cdot \Vert_{\Ht(\Omega)}} = \mathcal{H}_\theta(\Omega)$.
\end{enumerate}
\end{theorem}

\begin{proof}
Since $T$ is injective, the $\lVert \cdot \rVert_{TL_p(\Omega)}$ norms are well defined.
\begin{enumerate}[label=(\roman*)]
\item Due to the definition of the norm $\lVert \cdot \rVert_{TL_p(\Omega)}$, the completeness of $L_p(\Omega)$ carries over to $T(L_p(\Omega))$ for any $1 \leq p 
\leq 2$. 
For the continuous point evaluation, we consider $v \in L_p(\Omega)$ and $f = Tv \in T(L_p(\Omega))$ and use \eqref{eq:generalization_Lp} to estimate
\begin{align*}
|f(x)| &= |\langle f, k(\cdot, x) \rangle_{\ns}| = |\langle Tv, k(\cdot, x) \rangle_{\ns}| = |\langle v, k(\cdot, x) \rangle_{L_2(\Omega)}| \\
&\leq \Vert v \Vert_{L_p(\Omega)} \cdot \Vert k(\cdot, x) \Vert_{L_q(\Omega)} 
= \Vert Tv \Vert_{TL_p(\Omega)} \cdot \Vert k(\cdot, x) \Vert_{L_q(\Omega)} \\
&= \Vert f \Vert_{TL_p(\Omega)} \cdot \Vert k(\cdot, x) \Vert_{L_q(\Omega)}
\end{align*}
\item Let $Tv = f$ for some $v \in L^1(\Omega)$. 
Then we note that it holds by using \eqref{eq:dual_pairing_integral} twice:
\begin{align*}
\langle	f, \varphi_j \rangle_{L_2(\Omega)} &= \langle	 f, T\varphi_j \rangle_{\calh(\Omega)} = \lambda_j \langle	 f, \varphi_j \rangle_{\calh(\Omega)} = 
\lambda_j \langle	 Tv, \varphi_j \rangle_{\calh(\Omega)} \\
&= \lambda_j \langle	 v, \varphi_j \rangle_{L_2(\Omega)}.
\end{align*}
(This actually shows that we can also directly leverage $T$ being self-adjoint even though $\langle \cdot, \cdot \rangle_{L_2(\Omega)}$ is more of a notation rather than a real inner product.)

We consider the two cases $p=2$ and $p=1$:
\begin{itemize}
\item For $v \in L_2(\Omega)$ we compute:
\begin{align*}
\sum_{j=1}^\infty \frac{|\langle f, \varphi_j \rangle_{L_2(\Omega)}|^2}{\lambda_j^\theta}
&= \sum_{j=1}^\infty \lambda_j^{2-\theta} \cdot |\langle v, \varphi_j \rangle_{L_2(\Omega)}|^2 \\
&\leq \lambda_1^{2-\theta} \cdot \sum_{j=1}^\infty |\langle v, \varphi_j \rangle_{L_2(\Omega)}|^2 \\
&= \lambda_1^{2-\theta} \cdot \Vert v \Vert_{L_2(\Omega)}^2 \\
&= \lambda_1^{2-\theta} \cdot \Vert Tv \Vert_{TL_2(\Omega)}^2.
\end{align*}
This shows that $f = Tv \in \calh_\theta(\Omega)$ for any $\theta \leq 2$,
and thus in particular that $f \in \calh_2(\Omega)$.
\item For $v \in L_1(\Omega)$ we compute: 
\begin{align*}
\sum_{j=1}^\infty \frac{|\langle f, \varphi_j \rangle_{L_2(\Omega)}|^2}{\lambda_j^\theta}
&= \sum_{j=1}^\infty \lambda_j^{2-\theta} \cdot |\langle v, \varphi_j \rangle_{L_2(\Omega)}|^2 \\
&\leq \sum_{j=1}^\infty \lambda_j^{2-\theta} \cdot \left( \int_{\Omega} |v(x)|^{1/2} \cdot |v(x)|^{1/2} \cdot |\varphi_j(x)| ~ \mathrm{d}x \right)^2 \\
&\leq \Vert v \Vert_{L^1(\Omega)} \cdot \sum_{j=1}^\infty \lambda_j^{2-\theta} \int_{\Omega} |v(x)| \varphi_j(x)^2 ~ \mathrm{d}x \\
&\leq \Vert v \Vert_{L^1(\Omega)} \cdot \int_{\Omega} |v(x)| \cdot \sum_{j=1}^\infty \lambda_j^{2-\theta} \varphi_j(x)^2 ~ \mathrm{d}x \\
&\leq \sup_{x \in \Omega} \sum_{j=1}^\infty \lambda_j^{2-\theta} \varphi_j(x)^2  \cdot \Vert v \Vert_{L^1(\Omega)}^2 \\
&\leq \sup_{x \in \Omega} \sum_{j=1}^\infty \lambda_j^{2-\theta} \varphi_j(x)^2  \cdot \Vert Tv \Vert_{TL^1(\Omega)}^2.
\end{align*}
As $\calh(\Omega) \asymp H^\tau(\Omega)$ and $\calh_{2-\theta}(\Omega) \asymp H^{(2-\theta) \tau}(\Omega)$,
we have $\calh_{2-\theta}(\Omega) \asymp H^{(2-\theta) \tau}(\Omega) \hookrightarrow \mathcal{C}(\Omega)$ for $(2-\theta) \tau > d/2$.
Thus, by \cite[Theorem 5.3]{steinwart2012mercer} there is a constant $\kappa \geq 0$ 
such that $\sup_{x \in \Omega} \sum_{j=1}^\infty \lambda_j^{2-\theta} \varphi_j(x)^2 \leq \kappa^2$ almost everywhere.
Hence $f = Tv \in \calh_\theta(\Omega)$ for any $(2-\theta) \tau > d/2 \iff \theta < 2 - \frac{d}{2\tau}$.
\end{itemize}
This shows that $T \colon L_2(\Omega) \rightarrow \calh_2(\Omega)$ as well as $T \colon L_1(\Omega) \rightarrow \calh_\theta(\Omega)$ for $\theta < 2 - 
\frac{d}{2\tau}$ are bounded operators.
Using $\sigma \in [0, 1]$, the interpolation spaces between $L_1(\Omega)$ and $L_2(\Omega)$ are given as $L_{p_\sigma}(\Omega)$ with $\frac{1}{p_\sigma} = 
\frac{1-\sigma}{1} + \frac{\sigma}{2} = 1 - \frac{\sigma}{2} \iff \sigma = 2 - \frac{2}{p}$, 
while the interpolation spaces between $\calh_{2}(\Omega)$ and $\calh_{2-\frac{d}{2\tau}}(\Omega)$ are given as $\calh_{2 - \frac{d}{2\tau} + \sigma 
\frac{d}{2\tau}}(\Omega)$.
Thus interpolation theory yields that $T \colon L_p(\Omega) \rightarrow \calh_\theta(\Omega)$ is bounded for $\theta < 2 - \frac{2-p}{p} \cdot 
\frac{d}{2\tau}$. 
\item To prove the denseness of the embedding, we consider any $0 \leq \theta < 2 - \frac{2-p}{p} \cdot \frac{d}{2\tau}$ and an arbitrary $f \in \Ht(\Omega) 
\subseteq L_2(\Omega)$ such that $f \perp_{\Ht(\Omega)} T(L_p(\Omega))$, which is equivalent to $\langle f, Tv \rangle_{\Ht(\Omega)} = 0$ for all $v \in L_p(\Omega) 
\supseteq L_2(\Omega)$. 
We need to show that $f = 0$.
For this, we calculate that, for all $v \in L_p(\Omega)$,
\begin{align*}
0 = \langle f, Tv \rangle_{\Ht(\Omega)} = \sum_{j=1}^\infty \frac{\langle f, \varphi_j \rangle_{L_2(\Omega)} \langle Tv, \varphi_j \rangle_{L_2(\Omega)}}{\lambda_j^\theta} = \sum_{j=1}^\infty \frac{\langle f, \varphi_j \rangle_{L_2(\Omega)} \langle v, \varphi_j \rangle_{L_2(\Omega)}}{\lambda_j^{\theta-1}}.
\end{align*}
Since this holds for all $v \in L_p(\Omega)$, 
it needs to hold especially for $v = f$,
such that the previous equality turns into
\begin{align*}
0 = \sum_{j=1}^\infty \frac{|\langle f, \varphi_j \rangle_{L_2(\Omega)}|^2}{\lambda_j^{\theta-1}}.
\end{align*}
From this we conclude $\langle f, \varphi_j \rangle_{L_2(\Omega)} = 0$ for all $j \in \N$, and thus $f = 0$.

Now assume that $T(L_p(\Omega)) \hookrightarrow \mathcal{H}_\theta(\Omega)$ is surjective for some $0 \leq \theta^* < 2 - \frac{2-p}{p} \cdot 
\frac{d}{2\tau}$.
Then it follows that $\mathcal{H}_{\theta^*}(\Omega) = T(L_p(\Omega))$ as sets,
and in particular that
$\mathcal{H}_{\theta^*}(\Omega) \subset T(L_p(\Omega))$.
As both $T(L_p(\Omega))$ and $\mathcal{H}_{\theta^*}(\Omega)$ have continuous point evaluations,
\cite[Lemma 23]{scholpple2025spaces} yields the continuous embedding $\mathcal{H}_{\theta^*}(\Omega) \hookrightarrow T(L_p(\Omega))$.
Due to having embeddings in both directions, we can conclude that both the norms $\lVert \cdot \rVert_{T(L_p(\Omega))}$ and $\lVert \cdot 
\rVert_{\mathcal{H}_\theta(\Omega)}$ are equivalent.
This means that if we equip $T(L_p(\Omega))$ with the norm $\lVert \cdot \rVert_{\mathcal{H}_{\theta^*}(\Omega)}$, 
it would be a Hilbert space as well.
However, $(T(L_p(\Omega)), \lVert \cdot \rVert_{T(L_p(\Omega))})$ is isometrically isomorph to $(L_p(\Omega), \lVert \cdot \rVert_{L_p(\Omega)})$,
which is not a Hilbert space for $p \neq 2$.
As isometric isomorphisms preserve Hilbert space structures, this is a contradiction.
Thus $T(L_p(\Omega)) \hookrightarrow \mathcal{H}_\theta(\Omega)$ is not surjective for any $\theta^* < 2 - \frac{2-p}{p} \cdot \frac{d}{2\tau}$.
\item This is an immediate consequence from the dense embedding.
\end{enumerate}
\end{proof}

Without assuming specific properties on the operator $A$, 
a generalization of \Cref{th:embedding_power_spaces} from using the integral operator
\eqref{eq:T_operator} to using \eqref{eq:A_as_integral_operator} does not seem feasible.
When $A$ is a differentiation operator, it seems plausible that $A^* (L_p(\Omega))$ is related to a power space.
However, this is likely to require additional conditions that we have not been able to pinpoint.

\section{Boundary conditions}\label{sec:boundary}

We proved superconvergence for functions $v$ in the range $A^*(V')$ of the adjoint $A^*$ of a linear operator $A\in\call(\calh(\Omega), V)$, and this is in itself an exhaustive 
description of $v$.
However, as extensively discussed in~\cite{schaback2018superconvergence} there are many cases of interest where the condition $v\in A^*(V')$ can be formulated 
in terms of requirements of the smoothness of $v$ and on the realization of some boundary conditions. We are far from a complete systematization of this 
relation, but in this section we discuss some general principles and provide a few examples. 

We start by taking another look at the case of $A^*=T$ of Section~\ref{sec:superconv_mercer}, which is relevant also for the general case of 
Section~\ref{sec:superconv_in_images} thanks to Theorem~\ref{th:embedding_power_spaces}.

Consider a sufficiently regular set $\Omega\subset\R^d$ and a linear differential operator $D$ with suitable boundary conditions $B$ such 
that the problem
\begin{equation}\label{eq:pde_problem}
\begin{cases}
D v(x) = f(x),\;\;&x\in\Omega,\\
B v(x) = 0,\;\;&x\in\partial\Omega,
\end{cases}
\end{equation}
has a unique (weak or strong) solution given $f\in L_2(\Omega)$. If this PDE has a Green kernel $k$ that is symmetric and positive definite, then the solution 
is given by
\begin{equation}\label{eq:green_ker}
v(x) = \int_\Omega f(y) k(x, y) \mathrm{d} y.
\end{equation}
Since $k$ is symmetric and positive definite, we may also take the perspective of the previous sections and regard this expression as $v=T f=A^*f$, where $A\in\call(\calh(\Omega), L_2(\Omega))$ is the 
embedding operator. 
From the kernel side, we know that interpolating $v$ with $k$ will result in superconvergence since $v\in A^*(V')$. On the other hand, 
$A^*(V')$ is the space of solutions of the PDE~\eqref{eq:pde_problem} thanks to~\eqref{eq:green_ker}, and thus $v$ will have additional smoothness (depending on 
the order of $D$) and satisfy some boundary conditions (as defined by $B$).

Not all operators $A^*$ arise in this way as inverses of (pseudo)differential operators, and thus we should not expect to have in general a characterization of 
$A^*(V')$ in terms of smoothness and conditions on the boundary. 
For example, we may expect this to be the case for the general integral operators of Proposition~\ref{prop:A_as_integral_operator}, but not for the case of 
Example~\ref{ex:l1}. 

However, it turns out that this is always the case for Sobolev kernels and $A^*=T$, even if the precise definition of the differential operator and the 
boundary conditions in the PDE~\eqref{eq:pde_problem} depend not only on the equivalent Sobolev space $W_2^\tau(\Omega)$, but also on the specific 
(equivalent) norm of $\calh(\Omega)$.
Namely, Hubmer, Sherina and Ramlau~\cite{hubmer2023characterizations} consider the adjoint of embedding operators of Sobolev spaces into $L_2$, and observe that (in our notation) the 
equality~\eqref{eq:from_wendland_book} can be seen as a variational problem over $f\in\calh(\Omega)$, with $Tv$ as its unique solution. When $k$ is a 
Sobolev kernel and its inner product has a certain explicit form, this problem can be seen as the weak form of a PDE, which can be found by integrating by part 
this variational formulation and setting to zero the resulting boundary terms.
We refer to Section 3 of~\cite{hubmer2023characterizations} for further details on this construction, and report here a remarkable case.
\begin{prop}\label{prop:bvp}
Let $\Omega\subset\R^d$ be sufficiently smooth, $m\in\N$ with $m>d/2$, and let $k$ be the reproducing kernel of $W_2^m(\Omega)$ corresponding to the standard 
inner 
product~\eqref{eq:sobolev_norm}.
Then there are $m$ differential operators $N_i$ of order $i$ with $m\leq i\leq 2m -1$ such that, for each $f\in L_2(\Omega)$, $T f$ is the unique weak solution 
of the PDE
\begin{align*}
\sum_{|\alpha| \leq m} (-1)^{|\alpha|} D^{2\alpha} u(x) &= f(x), \;\;x\in\Omega,\\
N_{i} u(x) & = 0, \;\;x\in\partial \Omega,\;\; m\leq i\leq 2m-1.
\end{align*}
In particular, if $d=1$ and $\Omega=(a,b)$, $a,b\in\R$, then the boundary conditions are given by
\begin{equation*}
\sum_{j=i+1-m}^{m} (-1)^{j} D^{(2j-i-1+m)} u(x) = 0, \;\; x\in\{a, b\}, \;\; m\leq i\leq 2 m - 1.
\end{equation*}
\end{prop}
\begin{proof}
The general part for $d\geq1$ is essentially Proposition 3.3 in~\cite{hubmer2023characterizations}, which shows that if $A \colon W_2^m(\Omega)\to L_2(\Omega)$ is the 
embedding operator, then $A^* f$ solves the given PDE.

For $d=1$, the result follows from Proposition 1 in~\cite{saitoh2003reproducing}, which states that the reproducing kernel $k$ of $W^m_2(a, b)$ with its 
standard inner product is the Green kernel of the given PDE, from which the result follows by the discussion at the beginning of this section. We remark that 
the boundary conditions were given in~\cite{saitoh2003reproducing} (see equation (46) ) as
\begin{equation*}
\sum_{j=n}^m (-1)^j D^{(2j-n)} u(x) = 0, \;\; x\in\{a, b\}, \;\; n=1, 2, \dots, m,
\end{equation*}
which we rewrote to match the notation of the case $d>1$.
\end{proof}

We conclude this section with a number of remarks.

\begin{rem}
The paper~\cite{hubmer2023characterizations} provides similarly explicit results also for other equivalent norms on $W_2^m(\Omega)$, see for 
example Proposition 3.4 for the case for the equivalent norm $\norm{L_2(\Omega)}{u} + \seminorm{W^m_2(\Omega)}{u}$. 
Moreover, the case of $W_2^\tau(\R)$ with $\tau\notin\N$ is also addressed in~\cite{hubmer2023characterizations}, but only on the whole space.
\end{rem}

\begin{rem}
As anticipated in \Cref{sec:intro}, for Sobolev kernels we see that not all functions in a smoother Sobolev space can be approximated 
by a rougher kernel with optimal rates, differently to what happens to rough functions being approximated by smoother kernels 
(as happens in the misspecified setting, see~\cite{narcowich2006sobolev}). 
However, interpolating a smoother function with a correspondingly smoother kernel may lead to a worse stability, since the eigenvalues of the kernel matrix of a Sobolev kernel are known to decay according to the smoothness of the kernel. This may be regarded as an instance of a trade-off principle, which is common in kernel 
interpolation (see \cite{schaback1995error,schaback2023instabilities}). 
\end{rem}

\begin{rem}\label{rem:boundary_interpolation}
In the Sobolev case discussed here, one could expect that the interpolation space $(\calh(\Omega),A^*(V'))_{\theta, \infty}$ is a space of functions of 
corresponding 
intermediate smoothness between $A^*(V')$ and $\calh(\Omega)$, and satisfying only certain boundary conditions.
This then would also explain the additional $h^{1/2}$ superconvergence rates observed in e.g.\ \cite[Section 8]{schaback2018superconvergence} for interpolation using Sobolev kernels and 
even proven in \cite[Theorem 1.5]{johnson2001error} and \cite[Proposition 2]{lohndorf2017thin} for the case of surface spline interpolation.
Some experiments to test this hypothesis are reported in \Cref{subsec:sobolev_kernels} and \Cref{sec:ex_bcs}. 

Moreover, in this direction, we mention that Acquistapace and Terreni~\cite{acquistapace1987hoelder} show the following: Consider a PDE defined by a linear differential operator of 
order 
$2m\in\N$ and boundary conditions of order $m_1, \dots m_j$, and denote as $D$ the corresponding set of solutions. 
Then $(C(\overline \Omega),D)_{\theta, \infty}$ is precisely the space of functions of smoothness $2m\theta$ which satisfy only the boundary conditions up to 
smoothness $m_j\leq \lfloor 2m\theta\rfloor$ (see Theorem 2.3 in~\cite{acquistapace1987hoelder}).
We remark that the result requires a number of assumptions on $\Omega$ and on the operators, for which we refer to the original paper. The main difference 
with our scenario is the use of $C(\overline \Omega)$ in place of $\calh(\Omega)$.

Moreover, the interpolation setting is quite clear for some special differential operators when treated from the point of view of fractional powers of their 
spectrum, as in Section~\ref{sec:superconv_mercer}. For example,  we refer to \cite{musina2016fractional} for the case of the Laplacian.
\end{rem}

\subsection{Explicit examples in one dimension}\label{sec:green_1d}
We show the case of two symmetric and strictly positive definite kernels $k, k'$ whose RKHSs are equal to $W_2^1(0,1)$ with different but equivalent 
norms.
Moreover, they are Green kernels of the same differential operator, when 
completed to a uniquely solvable PDE by means of two different sets of boundary conditions. 
This shows in  particular an explicit example when the boundary conditions required for superconvergence are not determined up to norm equivalence.

Given $f\in L_2(0,1)$, for $\Omega=(0,1)$ and $d=m=1$ the PDE of Proposition~\ref{prop:bvp} is
\begin{equation}\label{eq:one-d-pde}
\begin{cases}
\begin{aligned}
-u''(x) + u(x) &= f(x), \;\; x\in(0,1),\\
u'(0)=u'(1)&=0,
\end{aligned}
\end{cases}
\end{equation}
and, according to Proposition 1 in~\cite{saitoh2003reproducing}, its Green kernel is the symmetric and strictly positive definite 
\begin{equation*}
k(x,y)\coloneqq \frac{\cosh(\min(x,y))\cosh(1-\max(x,y))}{\sinh(e)}.
\end{equation*}
Following \cite{saitoh2003reproducing} (see also Example 8.2 (a) in \cite{fasshauer2015kernel}), this is the reproducing kernel of $W_2^1(0,1)$ with the 
standard inner product
\begin{equation*}
\inner{\calh(k, (0,1))}{u,v} \coloneqq \inner{W_2^1(0,1)}{u,v}= \int_0^1 [u(x) v(x) + u'(x) v'(x) ] \mathrm{d} x.
\end{equation*}
If instead the boundary conditions in~\eqref{eq:one-d-pde} are replaced by $u'(0)-u(0)=u'(1)+u(1)=0$, then the Green kernel 
is the symmetric and strictly positive definite Mat\'ern kernel
\begin{equation*}
k'(x,y)\coloneqq \frac12 e^{-|x-y|},
\end{equation*}
as can be verified by direct computation. 
Following Theorem 1.3 in \cite{saitoh2016theory}, this is the reproducing kernel of $W_2^1(0,1)$ with the inner product
\begin{equation*}
\inner{\calh(k', (0,1))}{u,v}\coloneqq \int_0^1 [ u(x) v(x) + u'(x) v'(x) ] \mathrm{d} x + u(0)v(0) + u(1)v(1).
\end{equation*}

This means that $T(L_2(0,1))$ for the two kernels is the space of solutions of the two corresponding PDEs, meaning that this space contains 
functions with additional smoothness, but with different boundary conditions depending on the kernel.
Namely, in $\calh(k, (0,1))$ we have
\begin{equation*}
T(L_2(0,1)) = \{u\in W_2^2(0,1): u'(0)=u'(1) = 0\},
\end{equation*}
while in $\calh(k', (0,1))$ it is
\begin{equation*}
T(L_2(0,1)) = \{u\in W_2^2(0,1): u'(0)-u(0)=u'(1) +u(1)= 0\}.
\end{equation*}
To see that the two RKHSs are norm equivalent, observe that
\begin{equation*}
\norm{\calh(k, (0,1))}{u}^2
= \int_0^1 [ u^2(x) + (u'(x))^2 ] \mathrm{d} x
\leq \int_0^1 [ u^2(x) + (u'(x))^2 ] \mathrm{d} x + u^2(0) + u^2(1)
=\norm{\calh(k', (0,1))}{u}^2.
\end{equation*}
Moreover, for all $u\in W_2^1(0,1)$ and $x, x_0\in[0,1]$ we have
\begin{equation*}
|u(x_0)| \leq |u(x)| + \norm{L_2(0,1)}{u'} \sqrt{|x-x_0|},
\end{equation*}
which implies that
\begin{align*}
|u(x_0)|^2 
&= \int_0^1 |u(x_0)|^2 \mathrm{d} x 
\leq \int_0^1 |u(x)|^2  \mathrm{d} x + \norm{L_2(0,1)}{u'}^2 \int_0^1|x-x_0| \mathrm{d} x\\
&\leq \norm{L_2(0,1)}{u}^2 + \frac12\norm{L_2(0,1)}{u}^2
\leq \norm{W_2^1(0,1)}{u}.
\end{align*}
Thus
\begin{equation*}
\norm{\calh(k', (0,1))}{u}^2
= \int_0^1 [ u^2(x) + (u'(x))^2 ] \mathrm{d} x + u^2(0) + u^2(1)
\leq 3 \norm{W_2^1(0,1)}{u}
= 3 \norm{\calh(k, (0,1))}{u},
\end{equation*}
completing the norm equivalence.

\section{Numerical experiments and examples}\label{sec:experiments}

In the following we consider some examples that illustrate the continuous superconvergence bounds derived in the previous sections.

All experiments are defined on $\Omega=[0,1]^d$ for $d=1$ or $d=2$, where we consider Sobolev kernels and functions $f$ of various regularities. 
In all cases, given $n_0,n_{\max}\in\N$ we consider approximation of $f$ by the kernel interpolant, denoted $I_n f$ here, on $n=n_0, \dots, n_{\max}$ equally spaced points in the interior 
of $\Omega$. We omit the points on the boundary to avoid partially enforcing boundary conditions, which would sometimes result in polluted 
convergence rates. Observe that in this setting the fill distance scales as $h_X \asymp n^{-1/d}$.

The errors are evaluated on a finer grid of equally spaced points (now including the boundary), where the $L_1, L_2$ and $L_\infty$ errors $e_n = \lVert f - I_n f \rVert_{L_p(\Omega)}$ for $p \in \{1, 2, \infty\}$ are approximated in 
the standard way by their discrete counterparts. 
Theory predicts that the errors decay as $h_X^{a} \asymp n^{-a/d}$ for some convergence rate $a > 0$.
If $\{(n, e_n)\}_{n}$ is a sequence of such errors, the corresponding empirical convergence rate $a$ is computed from the slope of the linear fit on $\{(\log(n), \log(e_n))\}_{n}$.
The code to reproduce the results is available upon request.

\subsection{Sobolev kernels}
\label{subsec:sobolev_kernels}
We consider the basic, linear and quadratic Matérn kernels, which are radial basis function kernels given as
\begin{align}
\begin{aligned}
\label{eq:matern_kernels}
k_\text{basic}(x, z) &= \exp(-\Vert x - z \Vert) \\
k_\text{linear}(x, z) &= \exp(-\Vert x - z \Vert) \cdot (1 + \Vert x - z \Vert) \\
k_\text{quadratic}(x, z) &= \exp(-\Vert x - z \Vert) \cdot (3 + 3\Vert x - z \Vert + \Vert x - z \Vert^2).
\end{aligned}
\end{align}
For all these kernels we consider the approximation of $f_\alpha(x) = x_1^\alpha$ in $\Omega = [0, 1]^d$ for $d \in \{1, 2\}$.
The respective RKHS $\calh$ are normequivalent to $W_2^\tau(\Omega)$ for $\tau_\text{basic} = \frac{d+1}{2}$, $\tau_\text{linear} = \frac{d+3}{2}$ and $\tau_\text{quadratic} = \frac{d+5}{2}$.
Except for integer values of $\alpha$, the Sobolev smoothness of $f_\alpha$ scales according to $\alpha$, i.e.\ it holds
\begin{equation*}
f_\alpha \in W_2^{\sigma}(\Omega) \qquad \text{if and only if } \sigma < \alpha + 1/2.
\end{equation*}
However, as $f_\alpha$ does not satisfy any particular boundary conditions, we cannot expect $f_\alpha$ to be included in all the power spaces 
$\mathcal{H}_\theta(\Omega)$ for $\theta > 1$. 

The observed convergence rates are visualized in \Cref{fig:sobolev_conv_rates}.
One can observe that the $L_2(\Omega)$-convergence rates nicely scale according to the Sobolev smoothness $\alpha + 1/2$ up to $\tau + 1/2$, which corresponds to $\theta = 1 + 
\frac{1}{2\tau}$. 
Beyond, the convergence rates saturate.
Such a behaviour was already observed in \cite[Section]{schaback2018superconvergence}.
We point to \Cref{rem:boundary_interpolation} for an explanation of the additional rate of decay up to $h_X^{1/2}$.

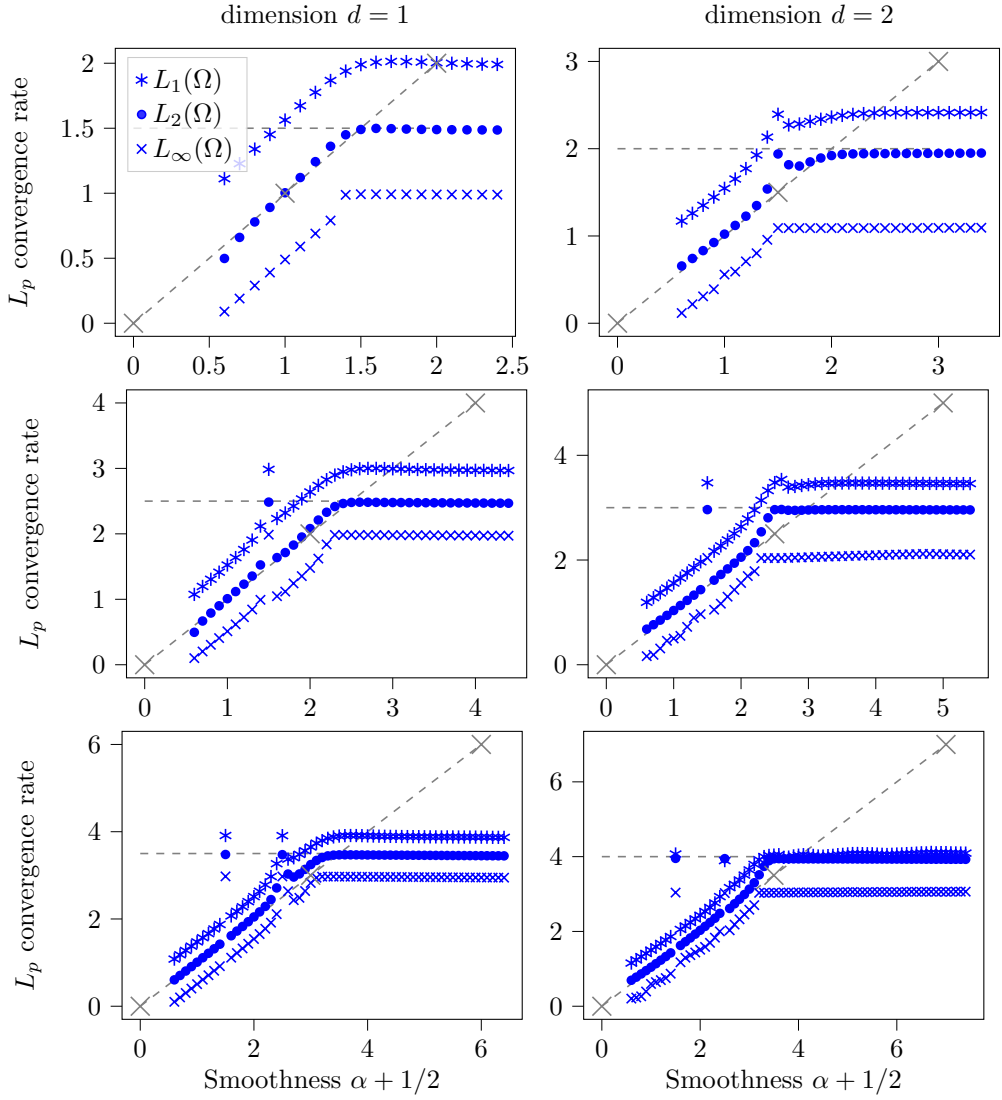
\begin{figure}[htp!]
\centering
\setlength\fwidth{.48\textwidth}
\begin{tikzpicture}

\definecolor{darkgray176}{RGB}{176,176,176}
\definecolor{gray}{RGB}{128,128,128}
\definecolor{lightgray204}{RGB}{204,204,204}

\begin{axis}[
width=0.951\fwidth,
height=0.75\fwidth,
at={(0\fwidth,0\fwidth)},
legend cell align={left},
legend style={
  fill opacity=0.8,
  draw opacity=1,
  text opacity=1,
  at={(0.03,0.97)},
  anchor=north west,
  draw=lightgray204
},
tick align=outside,
tick pos=left,
title={dimension $d=1$},
x grid style={darkgray176},
xmin=-0.12, xmax=2.52,
xtick style={color=black},
y grid style={darkgray176},
ylabel={$L_p$ convergence rate},
ymin=-0.100690288557874, ymax=2.11449605971536,
ytick style={color=black}
]
\addplot [semithick, blue, mark=asterisk, mark size=2.5, mark options={solid}, only marks]
table {%
0.6 1.1118629344269
0.7 1.22774286079347
0.8 1.33947439115909
0.9 1.45124630394285
1 1.56295219107124
1.1 1.67256992127859
1.2 1.77601993496688
1.3 1.86724027992785
1.4 1.93962779601117
1.5 1.98904076817102
1.6 2.00781423770332
1.7 2.01380577115749
1.8 2.01280351340302
1.9 2.0088546679555
2 2.00420949037945
2.1 1.99992023516537
2.2 1.99611402789013
2.3 1.9930319390239
2.4 1.99058408194657
};
\addlegendentry{$L_1(\Omega)$}
\addplot [semithick, blue, mark=*, mark size=1.5, mark options={solid}, only marks]
table {%
0.6 0.498207491357874
0.7 0.660211712971495
0.8 0.779121259941568
0.9 0.890904215264849
1 1.00344981515589
1.1 1.12059030216421
1.2 1.24293974786486
1.3 1.36110378800409
1.4 1.44964145287855
1.5 1.49033361492666
1.6 1.49873645531501
1.7 1.49671437684234
1.8 1.49350771559452
1.9 1.49112310287667
2 1.48954541369503
2.1 1.48846709134137
2.2 1.48765616557448
2.3 1.48697874842549
2.4 1.48636492870001
};
\addlegendentry{$L_2(\Omega)$}
\addplot [semithick, blue, mark=x, mark size=2.5, mark options={solid}, only marks]
table {%
0.6 0.0900038684716305
0.7 0.190003868471645
0.8 0.290003868471668
0.9 0.390003868471657
1 0.490003868471677
1.1 0.590003868471609
1.2 0.690003868471589
1.3 0.790003868471654
1.4 0.98759004389547
1.5 0.992519478799062
1.6 0.992137701967141
1.7 0.99174519625827
1.8 0.991343401720101
1.9 0.990933517585402
2 0.990516550590304
2.1 0.990093352107543
2.2 0.989664647003299
2.3 0.989231056292589
2.4 0.98879311509415
};
\addlegendentry{$L_\infty(\Omega)$}
\addplot [semithick, gray, dashed, mark=x, mark size=5, mark options={solid}, forget plot]
table {%
0 0
1 1
2 2
};
\addplot [semithick, gray, dashed, forget plot]
table {%
0 1.5
2 1.5
};
\end{axis}

\end{tikzpicture}
\begin{tikzpicture}

\definecolor{darkgray176}{RGB}{176,176,176}
\definecolor{gray}{RGB}{128,128,128}
\definecolor{lightgray204}{RGB}{204,204,204}

\begin{axis}[
width=0.951\fwidth,
height=0.75\fwidth,
at={(0\fwidth,0\fwidth)},
legend cell align={left},
legend style={
  fill opacity=0.8,
  draw opacity=1,
  text opacity=1,
  at={(0.03,0.97)},
  anchor=north west,
  draw=lightgray204
},
tick align=outside,
tick pos=left,
title={dimension $d=2$},
x grid style={darkgray176},
xmin=-0.17, xmax=3.57,
xtick style={color=black},
y grid style={darkgray176},
ymin=-0.15, ymax=3.15,
ytick style={color=black}
]
\addplot [semithick, blue, mark=asterisk, mark size=2.5, mark options={solid}, only marks]
table {%
0.6 1.17144988497847
0.7 1.26013095430508
0.8 1.35091660783731
0.9 1.445212272657
1 1.54612387528762
1.1 1.65294903209765
1.2 1.77268849978601
1.3 1.92745052704435
1.4 2.13289938010347
1.5 2.3946520418283
1.6 2.27258194308632
1.7 2.28355914316586
1.8 2.31340005922954
1.9 2.33761933613023
2 2.36307794562538
2.1 2.3814918748474
2.2 2.39375435531192
2.3 2.40276243172934
2.4 2.40829864330242
2.5 2.41080165937085
2.6 2.41200950392268
2.7 2.41263279164876
2.8 2.41300111288463
2.9 2.41317143288866
3 2.41338572189302
3.1 2.41365486411799
3.2 2.41394287790278
3.3 2.41430667659239
3.4 2.41478405517242
};
\addplot [semithick, blue, mark=*, mark size=1.5, mark options={solid}, only marks]
table {%
0.6 0.656264008954002
0.7 0.74227106653321
0.8 0.831966609202539
0.9 0.924943032670993
1 1.02108820495121
1.1 1.12093303153737
1.2 1.22682099780165
1.3 1.34831436703432
1.4 1.53791790173453
1.5 1.93963439387119
1.6 1.81630157552859
1.7 1.80147378415858
1.8 1.84884478870322
1.9 1.89340834218673
2 1.92147855415222
2.1 1.93517560479821
2.2 1.94066342317652
2.3 1.94262799222396
2.4 1.94342660905671
2.5 1.94396748466081
2.6 1.94452174456874
2.7 1.94513919740083
2.8 1.94581058599339
2.9 1.94651878190657
3 1.94725076977932
3.1 1.94799837594257
3.2 1.94875678565906
3.3 1.94952318958806
3.4 1.95029590585884
};
\addplot [semithick, blue, mark=x, mark size=2.5, mark options={solid}, only marks]
table {%
0.6 0.117092771349308
0.7 0.218236097682263
0.8 0.308262828292957
0.9 0.389588347607675
1 0.558073107190438
1.1 0.59181383913668
1.2 0.708844684304001
1.3 0.802578064858049
1.4 0.95600026998617
1.5 1.09083196729712
1.6 1.09065730186501
1.7 1.09053596139179
1.8 1.0904623437676
1.9 1.09043277213341
2 1.09044461839299
2.1 1.09049583295761
2.2 1.09058469415632
2.3 1.09070967583586
2.4 1.09086937786369
2.5 1.09106248991772
2.6 1.09128777257861
2.7 1.09154404714688
2.8 1.09183018980093
2.9 1.09214512792375
3 1.09248783723218
3.1 1.09285733941116
3.2 1.09325269987552
3.3 1.09367302551839
3.4 1.0941174626409
};
\addplot [semithick, gray, dashed, mark=x, mark size=5, mark options={solid}, forget plot]
table {%
0 0
1.5 1.5
3 3
};
\addplot [semithick, gray, dashed, forget plot]
table {%
0 2
3 2
};
\end{axis}

\end{tikzpicture}
\begin{tikzpicture}

\definecolor{darkgray176}{RGB}{176,176,176}
\definecolor{gray}{RGB}{128,128,128}
\definecolor{lightgray204}{RGB}{204,204,204}

\begin{axis}[
width=0.951\fwidth,
height=0.75\fwidth,
at={(0\fwidth,0\fwidth)},
legend cell align={left},
legend style={
  fill opacity=0.8,
  draw opacity=1,
  text opacity=1,
  at={(0.03,0.97)},
  anchor=north west,
  draw=lightgray204
},
tick align=outside,
tick pos=left,
x grid style={darkgray176},
xmin=-0.22, xmax=4.62,
xtick style={color=black},
y grid style={darkgray176},
ylabel={$L_p$ convergence rate},
ymin=-0.2, ymax=4.2,
ytick style={color=black}
]
\addplot [semithick, blue, mark=*, mark size=1.5, mark options={solid}, only marks]
table {%
0.6 0.49535808235919
0.7 0.668014328412458
0.8 0.789737589171719
0.9 0.900648692227114
1 1.00897384084901
1.1 1.11795443087499
1.2 1.23073460069045
1.3 1.35473888787832
1.4 1.52495123979597
1.5 2.48492514017152
1.6 1.63719672033793
1.7 1.71407304277045
1.8 1.82839617972049
1.9 1.95130039497548
2 2.08024603368703
2.1 2.21073772569871
2.2 2.32985401716327
2.3 2.41758534117403
2.4 2.46385977523157
2.5 2.48019567684186
2.6 2.48334355745422
2.7 2.48250958614308
2.8 2.48097406898393
2.9 2.47958984348022
3 2.47846180690715
3.1 2.4775237360184
3.2 2.47670004902126
3.3 2.47593785417483
3.4 2.47520547520912
3.5 2.47448523117876
3.6 2.47376781322666
3.7 2.47304851075866
3.8 2.47232504924393
3.9 2.47159649621131
4 2.47086255330593
4.1 2.47012325099701
4.2 2.46937876136373
4.3 2.46862934356466
4.4 2.46787525559487
};
\addplot [semithick, blue, mark=x, mark size=2.5, mark options={solid}, only marks]
table {%
0.6 0.101039259680148
0.7 0.202644003453403
0.8 0.304706563518421
0.9 0.407448865072527
1 0.511263928696779
1.1 0.616920475144786
1.2 0.726152787229605
1.3 0.84387751956829
1.4 0.991789064721972
1.5 1.98741681870213
1.6 1.04744990823613
1.7 1.11797007168738
1.8 1.24269423921848
1.9 1.35433696758604
2 1.47844915060969
2.1 1.62191528828081
2.2 1.83819838949396
2.3 1.98315110332792
2.4 1.98380627174541
2.5 1.98328161590887
2.6 1.98273923638942
2.7 1.98218061493877
2.8 1.98160711901948
2.9 1.98101998765275
3 1.98042035424339
3.1 1.97980922876089
3.2 1.97918753698091
3.3 1.97855609588036
3.4 1.97791569607407
3.5 1.97726699759868
3.6 1.97661061477981
3.7 1.97594713990977
3.8 1.97527706094092
3.9 1.97460086101342
4 1.97391896824819
4.1 1.97323177508999
4.2 1.97253963461766
4.3 1.9718428848774
4.4 1.97114180166177
};
\addplot [semithick, blue, mark=asterisk, mark size=2.5, mark options={solid}, only marks]
table {%
0.6 1.07461172759958
0.7 1.19625964996525
0.8 1.30601546794943
0.9 1.41412148650287
1 1.5236414884806
1.1 1.63736829314943
1.2 1.76031776062992
1.3 1.90513070102785
1.4 2.11768062310865
1.5 2.98505303859043
1.6 2.23384115201772
1.7 2.31965956921788
1.8 2.42444234205506
1.9 2.53321083347468
2 2.64027954055666
2.1 2.74041024621175
2.2 2.82797772080136
2.3 2.89793989400651
2.4 2.94782145039572
2.5 2.97873885693421
2.6 2.99467485673703
2.7 3.00055210599904
2.8 3.0006053149772
2.9 2.99780267916963
3 2.99393869434028
3.1 2.98997346760838
3.2 2.98634397544082
3.3 2.98320525526334
3.4 2.9805570167558
3.5 2.97834385751867
3.6 2.97648652085469
3.7 2.97490877254937
3.8 2.97354806636274
3.9 2.97234714919951
4 2.97126309786535
4.1 2.97026335989869
4.2 2.96932334957376
4.3 2.96842403560283
4.4 2.96755601430118
};
\addplot [semithick, gray, dashed, mark=x, mark size=5, mark options={solid}, forget plot]
table {%
0 0
2 2
4 4
};
\addplot [semithick, gray, dashed, forget plot]
table {%
0 2.5
4 2.5
};
\end{axis}

\end{tikzpicture}
\begin{tikzpicture}

\definecolor{darkgray176}{RGB}{176,176,176}
\definecolor{gray}{RGB}{128,128,128}
\definecolor{lightgray204}{RGB}{204,204,204}

\begin{axis}[
width=0.951\fwidth,
height=0.75\fwidth,
at={(0\fwidth,0\fwidth)},
legend cell align={left},
legend style={
  fill opacity=0.8,
  draw opacity=1,
  text opacity=1,
  at={(0.03,0.97)},
  anchor=north west,
  draw=lightgray204
},
tick align=outside,
tick pos=left,
x grid style={darkgray176},
xmin=-0.27, xmax=5.67,
xtick style={color=black},
y grid style={darkgray176},
ymin=-0.25, ymax=5.25,
ytick style={color=black}
]
\addplot [semithick, blue, mark=asterisk, mark size=2.5, mark options={solid}, only marks]
table {%
0.6 1.19220802644604
0.7 1.28137078963965
0.8 1.3723845899683
0.9 1.46507695219292
1 1.55933466299215
1.1 1.65515668158755
1.2 1.75289410986466
1.3 1.85396134496279
1.4 1.96619055603522
1.5 3.47577711904289
1.6 2.16674537852645
1.7 2.27028440478317
1.8 2.38343522074156
1.9 2.50391474884524
2 2.63612652451196
2.1 2.78460246117121
2.2 2.95346939002897
2.3 3.14028015415817
2.4 3.3267152090268
2.5 3.47921957198077
2.6 3.53834512587331
2.7 3.39238529927322
2.8 3.40492539319665
2.9 3.423218401439
3 3.43950180937755
3.1 3.45208526217281
3.2 3.46042476484761
3.3 3.46574812144603
3.4 3.46877262557868
3.5 3.4703346899907
3.6 3.47079234203877
3.7 3.47059983207097
3.8 3.47002943157981
3.9 3.46925419265937
4 3.46842571180203
4.1 3.46754859292849
4.2 3.46673650908174
4.3 3.46598624665887
4.4 3.46532683447715
4.5 3.46471732588567
4.6 3.46411859145429
4.7 3.46351444779032
4.8 3.46291578678892
4.9 3.46231119111209
5 3.46172929199061
5.1 3.46114270155185
5.2 3.46056953699117
5.3 3.46003165368869
5.4 3.45948906111632
};
\addplot [semithick, blue, mark=*, mark size=1.5, mark options={solid}, only marks]
table {%
0.6 0.679950667434329
0.7 0.764264662091859
0.8 0.85203920798092
0.9 0.942768621986263
1 1.03602085139563
1.1 1.13147165141408
1.2 1.2289754907724
1.3 1.32882705475189
1.4 1.43359605743788
1.5 2.96227078088163
1.6 1.6150493877734
1.7 1.72351496183542
1.8 1.83058910530011
1.9 1.9398539842613
2 2.05419700652961
2.1 2.17952745773987
2.2 2.33047499405546
2.3 2.53851742588579
2.4 2.80638976464206
2.5 2.961965295063
2.6 2.965308478663
2.7 2.95051933138685
2.8 2.94824400234104
2.9 2.95175151367988
3 2.95570710067717
3.1 2.95844708579252
3.2 2.95995856022936
3.3 2.96064312076356
3.4 2.96086217846632
3.5 2.96084682697964
3.6 2.9607228768083
3.7 2.96055233135742
3.8 2.96036341621482
3.9 2.96016815108658
4 2.95997119850147
4.1 2.9597742274864
4.2 2.95957754548067
4.3 2.95938137802691
4.4 2.95918555083404
4.5 2.95899003208312
4.6 2.95879471495325
4.7 2.95859967491423
4.8 2.95840487011675
4.9 2.9582102456467
5 2.95801585881332
5.1 2.9578217073982
5.2 2.95762778634123
5.3 2.95743405203896
5.4 2.95724056264152
};
\addplot [semithick, blue, mark=x, mark size=2.5, mark options={solid}, only marks]
table {%
0.6 0.163748479951399
0.7 0.189249517508107
0.8 0.310540598347274
0.9 0.45774397430045
1 0.500439910026045
1.1 0.551436021203
1.2 0.723300124786996
1.3 0.895150957055159
1.4 0.96582612927566
1.5 2.04049454659814
1.6 1.0509747395857
1.7 1.16593155114332
1.8 1.29972772774651
1.9 1.43092099053323
2 1.56172195634812
2.1 1.69351298583313
2.2 1.78650737385134
2.3 2.03728180655919
2.4 2.03717952118508
2.5 2.03711582708515
2.6 2.03709749302637
2.7 2.0407279440727
2.8 2.04494178624215
2.9 2.04911108293353
3 2.05323635046612
3.1 2.05731749834401
3.2 2.06135464281432
3.3 2.06534773157598
3.4 2.06929698360077
3.5 2.0732020542874
3.6 2.07706329671054
3.7 2.08088069523286
3.8 2.08465458426749
3.9 2.08838541417462
4 2.09207343711892
4.1 2.09571896118093
4.2 2.09932262881412
4.3 2.10288493192615
4.4 2.10640639108786
4.5 2.10988751679175
4.6 2.1133286899445
4.7 2.11673082256161
4.8 2.11631111966295
4.9 2.11363485102832
5 2.11103623460122
5.1 2.1085128244857
5.2 2.10606205686378
5.3 2.10368164476613
5.4 2.10385544433356
};
\addplot [semithick, gray, dashed, mark=x, mark size=5, mark options={solid}, forget plot]
table {%
0 0
2.5 2.5
5 5
};
\addplot [semithick, gray, dashed, forget plot]
table {%
0 3
5 3
};
\end{axis}

\end{tikzpicture}
\begin{tikzpicture}

\definecolor{darkgray176}{RGB}{176,176,176}
\definecolor{gray}{RGB}{128,128,128}
\definecolor{lightgray204}{RGB}{204,204,204}

\begin{axis}[
width=0.951\fwidth,
height=0.75\fwidth,
at={(0\fwidth,0\fwidth)},
legend cell align={left},
legend style={
  fill opacity=0.8,
  draw opacity=1,
  text opacity=1,
  at={(0.03,0.97)},
  anchor=north west,
  draw=lightgray204
},
tick align=outside,
tick pos=left,
x grid style={darkgray176},
xlabel={Smoothness $\alpha+1/2$},
xmin=-0.32, xmax=6.72,
xtick style={color=black},
y grid style={darkgray176},
ylabel={$L_p$ convergence rate},
ymin=-0.3, ymax=6.3,
ytick style={color=black}
]
\addplot [semithick, blue, mark=*, mark size=1.5, mark options={solid}, only marks]
table {%
0.6 0.604777742587994
0.7 0.705580546005859
0.8 0.806632057198667
0.9 0.907902178983624
1 1.00940094501817
1.1 1.11113928178574
1.2 1.2133643416141
1.3 1.31634006047332
1.4 1.42192312835476
1.5 3.4779243581519
1.6 1.61515547701962
1.7 1.72245190808217
1.8 1.82861770272225
1.9 1.93633599516491
2 2.04648942214963
2.1 2.1618337780687
2.2 2.28818783033524
2.3 2.44448897304304
2.4 2.71166338076895
2.5 3.47428175019614
2.6 3.02975662050654
2.7 2.96117495265338
2.8 3.03229360011394
2.9 3.137449434747
3 3.24660816759262
3.1 3.34137755315984
3.2 3.40822614052149
3.3 3.44618840255408
3.4 3.4632549781133
3.5 3.46917524469038
3.6 3.47029674556989
3.7 3.46973966945887
3.8 3.46874357756638
3.9 3.46773886599407
4 3.46663586390003
4.1 3.46595792316334
4.2 3.46499597817776
4.3 3.46441065297587
4.4 3.46348312011638
4.5 3.46284979621376
4.6 3.46208037146019
4.7 3.46131299103459
4.8 3.4604944229713
4.9 3.45962611478391
5 3.45886419849351
5.1 3.45803686235517
5.2 3.45711277456188
5.3 3.45639783800282
5.4 3.45551314583535
5.5 3.45460633954837
5.6 3.45372830645075
5.7 3.45293335395198
5.8 3.4520322030162
5.9 3.45111728700449
6 3.45022492160704
6.1 3.44945047313482
6.2 3.44859175004518
6.3 3.44766125974819
6.4 3.44655857273417
};
\addplot [semithick, blue, mark=x, mark size=2.5, mark options={solid}, only marks]
table {%
0.6 0.100964534277553
0.7 0.202225870298186
0.8 0.303348812113074
0.9 0.404699215274054
1 0.506400787501303
1.1 0.607759530924507
1.2 0.710232461640476
1.3 0.813005100586609
1.4 0.917969599115794
1.5 2.97622016291701
1.6 1.11273603318875
1.7 1.2190974745437
1.8 1.32496489690841
1.9 1.43232759761677
2 1.54111334522374
2.1 1.65449136234492
2.2 1.77627192845728
2.3 1.91483135123237
2.4 2.11089120373219
2.5 2.97430433721763
2.6 2.63606629010622
2.7 2.4314255213407
2.8 2.48447615321668
2.9 2.64049101161546
3 2.82546936956149
3.1 2.94554984727683
3.2 2.97063153618353
3.3 2.97028168380005
3.4 2.96999236144971
3.5 2.96931674707056
3.6 2.96861379063236
3.7 2.96816571930879
3.8 2.96741073861197
3.9 2.96677141322939
4 2.96593349597381
4.1 2.96553552632867
4.2 2.96456725108225
4.3 2.96428048226044
4.4 2.963092168815
4.5 2.96284399779514
4.6 2.9619362439911
4.7 2.9612314132848
4.8 2.96035804886476
4.9 2.95957197246804
5 2.95848030996794
5.1 2.95807020259712
5.2 2.95712012623808
5.3 2.95649906659575
5.4 2.95549467008895
5.5 2.9544779118286
5.6 2.95373451022109
5.7 2.95294692810395
5.8 2.95184845738805
5.9 2.9511108860122
6 2.95058256993299
6.1 2.94961249494605
6.2 2.94859863095228
6.3 2.94788806159261
6.4 2.94672861858662
};
\addplot [semithick, blue, mark=asterisk, mark size=2.5, mark options={solid}, only marks]
table {%
0.6 1.08031610092522
0.7 1.17944793334971
0.8 1.27825064665892
0.9 1.37678962664921
1 1.47452614304121
1.1 1.57037261276123
1.2 1.66406711124404
1.3 1.77099303789604
1.4 1.87340129806882
1.5 3.91258700064986
1.6 2.06976260761801
1.7 2.17507421351221
1.8 2.27906519813361
1.9 2.39723488257166
2 2.51197357827741
2.1 2.63414942471122
2.2 2.78140705482172
2.3 2.97254552455843
2.4 3.26699380566214
2.5 3.90979403691354
2.6 3.3712990710706
2.7 3.40332623842991
2.8 3.48100532340604
2.9 3.56237911962006
3 3.64866768643988
3.1 3.73785987015286
3.2 3.79620687326066
3.3 3.8490817341286
3.4 3.8802235452541
3.5 3.90036111183503
3.6 3.90992663717478
3.7 3.91288460178154
3.8 3.9119784283329
3.9 3.90992445993392
4 3.90425221094663
4.1 3.90150052170192
4.2 3.89778638207049
4.3 3.89549707893134
4.4 3.89364504122213
4.5 3.89202359218497
4.6 3.88969640734448
4.7 3.88818031985397
4.8 3.88744829525831
4.9 3.8865480486999
5 3.88420097538626
5.1 3.88255489142769
5.2 3.88173032838889
5.3 3.87994213643951
5.4 3.88091956836485
5.5 3.88009086374311
5.6 3.87849340077465
5.7 3.87854631263451
5.8 3.87777644597808
5.9 3.87678923694218
6 3.87555000774989
6.1 3.87501802386281
6.2 3.87129404813939
6.3 3.87118035354093
6.4 3.87015996869155
};
\addplot [semithick, gray, dashed, mark=x, mark size=5, mark options={solid}, forget plot]
table {%
0 0
3 3
6 6
};
\addplot [semithick, gray, dashed, forget plot]
table {%
0 3.5
6 3.5
};
\end{axis}

\end{tikzpicture}
\begin{tikzpicture}

\definecolor{darkgray176}{RGB}{176,176,176}
\definecolor{gray}{RGB}{128,128,128}
\definecolor{lightgray204}{RGB}{204,204,204}

\begin{axis}[
width=0.951\fwidth,
height=0.75\fwidth,
at={(0\fwidth,0\fwidth)},
legend cell align={left},
legend style={
  fill opacity=0.8,
  draw opacity=1,
  text opacity=1,
  at={(0.03,0.97)},
  anchor=north west,
  draw=lightgray204
},
tick align=outside,
tick pos=left,
x grid style={darkgray176},
xlabel={Smoothness $\alpha+1/2$},
xmin=-0.37, xmax=7.77,
xtick style={color=black},
y grid style={darkgray176},
ymin=-0.35, ymax=7.35,
ytick style={color=black}
]
\addplot [semithick, blue, mark=asterisk, mark size=2.5, mark options={solid}, only marks]
table {%
0.6 1.15431235556504
0.7 1.23903296211342
0.8 1.32520078421938
0.9 1.41210511972869
1 1.50057504219799
1.1 1.58923436580903
1.2 1.67977687288086
1.3 1.7722923734116
1.4 1.86468939953023
1.5 4.06999713852756
1.6 2.05304842447623
1.7 2.14629260187653
1.8 2.24056830876046
1.9 2.33745112107786
2 2.43385185828405
2.1 2.53564192510558
2.2 2.64242394655552
2.3 2.7625735386598
2.4 2.93067443836221
2.5 3.89435207657146
2.6 3.17118743687624
2.7 3.26104193239607
2.8 3.37733292618256
2.9 3.50671372606654
3 3.64591407444857
3.1 3.79249319722737
3.2 3.91731802485331
3.3 4.00244257334348
3.4 4.0330322971265
3.5 4.04092591627834
3.6 4.05231912325195
3.7 3.98116155286914
3.8 3.9877159573644
3.9 4.01977503816522
4 4.03887171120406
4.1 4.05215604925741
4.2 4.04258766247814
4.3 4.033045222908
4.4 4.01815211736899
4.5 4.01346995439465
4.6 4.02713644335221
4.7 4.03499234030851
4.8 4.05265079044164
4.9 4.06107429535371
5 4.0737322151642
5.1 4.08298014279335
5.2 4.08458525781747
5.3 4.08003477915101
5.4 4.07787629499789
5.5 4.07694688661275
5.6 4.06889710260195
5.7 4.06250108808298
5.8 4.0574825005761
5.9 4.06274884310099
6 4.06983515733199
6.1 4.07223189126723
6.2 4.08166671111866
6.3 4.08536342950199
6.4 4.09429079803354
6.5 4.10276162638973
6.6 4.10066893567032
6.7 4.11256554588363
6.8 4.10255903407394
6.9 4.10788570944888
7 4.10771947304035
7.1 4.10175253757711
7.2 4.09888477149967
7.3 4.09708977257361
7.4 4.09405922599726
};
\addplot [semithick, blue, mark=*, mark size=1.5, mark options={solid}, only marks]
table {%
0.6 0.694702041650881
0.7 0.777357512114509
0.8 0.863581181002932
0.9 0.952392835147104
1 1.0440233639665
1.1 1.13741568343086
1.2 1.23259821700035
1.3 1.32885441301268
1.4 1.42665122003106
1.5 3.95367557336499
1.6 1.62459395752245
1.7 1.72491441470584
1.8 1.82585032632793
1.9 1.92799653027889
2 2.03064491287331
2.1 2.13572858494462
2.2 2.24203894284311
2.3 2.35351849515322
2.4 2.48392937207173
2.5 3.94596899699261
2.6 2.61396945586331
2.7 2.73965954222462
2.8 2.86106217491698
2.9 2.98813280558437
3 3.13110142612323
3.1 3.30227176481873
3.2 3.51413550118215
3.3 3.73875029426627
3.4 3.89092510403254
3.5 3.94191062189347
3.6 3.94650944939542
3.7 3.9438268396546
3.8 3.9404949309058
3.9 3.93999900459822
4 3.93943576759246
4.1 3.93995220537294
4.2 3.93933571125451
4.3 3.93943747867113
4.4 3.93900733110524
4.5 3.93857057191674
4.6 3.93979809151988
4.7 3.93839031088774
4.8 3.93992876669357
4.9 3.93787617169115
5 3.9377672511497
5.1 3.93761557094634
5.2 3.93674963772044
5.3 3.93556330214462
5.4 3.93536985397372
5.5 3.93580367145807
5.6 3.9356951009167
5.7 3.93560296066124
5.8 3.93519213221326
5.9 3.9343280696975
6 3.93479388181645
6.1 3.93316228774827
6.2 3.93329393743155
6.3 3.93223843650589
6.4 3.93347094976979
6.5 3.93265127176836
6.6 3.93092676582179
6.7 3.93185300729123
6.8 3.93000290402277
6.9 3.9319033728328
7 3.93007563292544
7.1 3.92914761270558
7.2 3.92944633303573
7.3 3.92981678275981
7.4 3.92811446942759
};
\addplot [semithick, blue, mark=x, mark size=2.5, mark options={solid}, only marks]
table {%
0.6 0.203953854981457
0.7 0.233452589659879
0.8 0.266112339250942
0.9 0.39417077190402
1 0.596390634462387
1.1 0.645561907705094
1.2 0.69855907891247
1.3 0.753056214364526
1.4 0.869675851712737
1.5 3.0391443432426
1.6 1.17573368498984
1.7 1.29769088750249
1.8 1.3666866626426
1.9 1.43708508110448
2 1.51083339584386
2.1 1.58660681455621
2.2 1.6949212419233
2.3 1.83338794745559
2.4 2.0018312408011
2.5 3.02588745059942
2.6 2.03043741429858
2.7 2.18403787942762
2.8 2.32464432593491
2.9 2.45381503193687
3 2.58071559911213
3.1 2.70305756026205
3.2 3.03186850041327
3.3 3.03360923511516
3.4 3.03308038861604
3.5 3.03763802669136
3.6 3.03286474656651
3.7 3.04056407341491
3.8 3.04033492706669
3.9 3.03730674505021
4 3.04607986022859
4.1 3.04235644846128
4.2 3.04430456812711
4.3 3.04392681444888
4.4 3.04530429873061
4.5 3.04730702022197
4.6 3.04545394137232
4.7 3.04469772849966
4.8 3.04604333797034
4.9 3.04872356737574
5 3.0509290629415
5.1 3.05493991914148
5.2 3.04959268744169
5.3 3.0479609709294
5.4 3.05061977217756
5.5 3.05337148252589
5.6 3.05398631066749
5.7 3.05602231453277
5.8 3.05586825519346
5.9 3.05377109697465
6 3.05145764666095
6.1 3.05690515023033
6.2 3.05822093606826
6.3 3.06151517817687
6.4 3.05976896810207
6.5 3.05847361272665
6.6 3.06043457367336
6.7 3.06274878810136
6.8 3.05818042215913
6.9 3.06197826516113
7 3.05930092363563
7.1 3.05952440910195
7.2 3.06557536303163
7.3 3.06383958087
7.4 3.06387144230314
};
\addplot [semithick, gray, dashed, mark=x, mark size=5, mark options={solid}, forget plot]
table {%
0 0
3.5 3.5
7 7
};
\addplot [semithick, gray, dashed, forget plot]
table {%
0 4
7 4
};
\end{axis}

\end{tikzpicture}
\caption{
Visualization of the convergence rates for interpolation error of $f_\alpha$ ($y$-axis) in the $L_1(\Omega)$, $L_2(\Omega)$, $L_\infty(\Omega)$ norm 
in terms of the fill distance $h_{X}$ over the parameter $\alpha$ ($x$-axis) for $\Omega = [0, 1]^d$, $d=1$ (left) and $d=2$ (right).
From top to bottom: Basic, linear and quadratic Matérn kernel (see \eqref{eq:matern_kernels}).
}
\label{fig:sobolev_conv_rates}
\end{figure}

\subsection{Saturation: No further superconvergence beyond $\theta = 2$}

We consider the piecewise linear Wendland
kernel \begin{align*}
k(x,y) = \max(1 - |x-y|, 0),
\end{align*}
defined on the domain $\Omega = [0,1]$. 
Its RKHS $\calh(\Omega)$ is norm equivalent to the Sobolev space $W_2^1(\Omega)$, i.e.\ $\tau = 1$. 
It is immediate to see that kernel interpolation using this kernel on pairwise distinct centers $X \subset \Omega$ boils down to piecewise linear spline 
interpolation on the subinterval $[\min X, \max X] \subset \Omega$. 

According to \Cref{subsec:superconv_kernel} we can expect a $\lVert \cdot \rVert_{L_2(\Omega)}$ convergence rate as $h_{X}^{2} \asymp |X|^{-2}$ for functions $f 
\in T(L_2(\Omega))$.
On the other hand, from the theory of linear spline interpolation we obtain the lower bound 
\begin{align}
\label{eq:lower_bound_spline}
\Vert f - I_X f \Vert_{L_2(\Omega)} &\geq c \cdot |X|^{-2}
\end{align}
for functions $f \in C([0, 1]) \cap C^2((0, 1))$ with $f''$ strictly larger zero on $[0, 1]$,
which can be generalized to functions with $f'' > 0$ and $f'' < 0$ piecewise.
Thus, the best possible $L^2(\Omega)$ convergence rate is already obtained for $f \in T(L_2(\Omega))$, 
i.e.\ one cannot expect in general any further improvement of the convergence rate for $f \in \mathcal{H}_{\theta}(\Omega)$, $\theta > 2$.

As a numerical experiment, we consider the interpolation of the function $f_\alpha(x) = x^\alpha$.
In this case both boundary points are included,
which allows to satisfy boundary conditions of low order.
The results are depicted in \Cref{fig:no_further_superconv}: 
Due to the use of both boundary points and the low smoothness of the kernel, 
the boundary conditions are satisfied explicitly, such that the $L_2(\Omega)$ convergence rate directly scales according to the Sobolev smoothness of the 
function $f_\alpha$ up to 2.

Due to $\lVert \cdot \rVert_{L_\infty(\Omega)} \geq \lVert \cdot \rVert_{L_2(\Omega)}$, 
also the $L_\infty(\Omega)$ error is lower bounded due to \eqref{eq:lower_bound_spline}.
Interestingly, this maximal rate of $L_1(\Omega)$ convergence is obtain for an even higher smoothness compared to the $L_2(\Omega)$ convergence rate.

\begin{figure}[htp!]
\centering
\setlength\fwidth{.48\textwidth}
\begin{tikzpicture}

\definecolor{darkgray176}{RGB}{176,176,176}
\definecolor{gray}{RGB}{128,128,128}
\definecolor{lightgray204}{RGB}{204,204,204}

\begin{axis}[
legend cell align={left},
legend style={
  fill opacity=0.8,
  draw opacity=1,
  text opacity=1,
  at={(0.03,0.97)},
  anchor=north west,
  draw=lightgray204
},
tick align=outside,
tick pos=left,
x grid style={darkgray176},
xlabel={Smoothness $\alpha + 1/2$},
xmin=-0.12, xmax=2.52,
xtick style={color=black},
y grid style={darkgray176},
ylabel={$L_p$ convergence rate},
ymin=-0.186244108017403, ymax=2.10410686228654,
ytick style={color=black}
]
\addplot [semithick, blue, mark=asterisk, mark size=2.5, mark options={solid}, only marks]
table {%
0.6 1.10157887832526
0.7 1.20025859086752
0.8 1.29849644347069
0.9 1.39564523662444
1 1.49074497322909
1.1 1.5823914814965
1.2 1.66866217112901
1.3 1.74720232828288
1.4 1.81558002373095
1.5 -0.00145796496968987
1.6 1.9154556372716
1.7 1.94697282196953
1.8 1.96835252864869
1.9 1.98201421068527
2 1.99028542538591
2.1 1.99505144975485
2.2 1.99766898979592
2.3 1.99903430395787
2.4 1.99970245956371
};
\addlegendentry{$L_1(\Omega)$}
\addplot [semithick, blue, mark=*, mark size=1.5, mark options={solid}, only marks]
table {%
0.6 0.60096274021122
0.7 0.700423343454756
0.8 0.800187555056696
0.9 0.900083732918121
1 1.00003687286456
1.1 1.10001300449654
1.2 1.19999267201368
1.3 1.29995274348248
1.4 1.39983838195783
1.5 0.0148028647260693
1.6 1.59846588467475
1.7 1.69554154314681
1.8 1.78775868024017
1.9 1.8692131394005
2 1.93182428005786
2.1 1.97083048611587
2.2 1.9898478005156
2.3 1.99718533493797
2.4 1.9994612269476
};
\addlegendentry{$L_2(\Omega)$}
\addplot [semithick, blue, mark=x, mark size=2.5, mark options={solid}, only marks]
table {%
0.6 0.100006049479414
0.7 0.200012302220852
0.8 0.300014989072727
0.9 0.400011761213616
1 0.500017221339936
1.1 0.600012230818193
1.2 0.700005811254762
1.3 0.800022497378773
1.4 0.900008137254211
1.5 -0.0821372457308601
1.6 1.10001002950207
1.7 1.20001078790244
1.8 1.30001344396967
1.9 1.40001038256294
2 1.50000446223593
2.1 1.60000291236285
2.2 1.70001431085894
2.3 1.80002048307915
2.4 1.90000113464016
};
\addlegendentry{$L_\infty(\Omega)$}
\addplot [semithick, gray, dashed, mark=x, mark size=5, mark options={solid}, forget plot]
table {%
0 0
1 1
2 2
};
\addplot [semithick, gray, dashed, forget plot]
table {%
0 0
2 0
};
\addplot [semithick, gray, dashed, forget plot]
table {%
2 0
2 2
};
\end{axis}

\end{tikzpicture}
\caption{
Visualization of the convergence rates for interpolation error of $f_\alpha$ ($y$-axis) in the $L_1(\Omega), L_2(\Omega), L_\infty(\Omega)$ norms in 
terms of the fill distance $h_{X}$ over the parameter $\alpha$ ($x$-axis) for $\Omega = [0, 1]$ and a Wendland kernel of low smoothness.
}
\label{fig:no_further_superconv}
\end{figure}
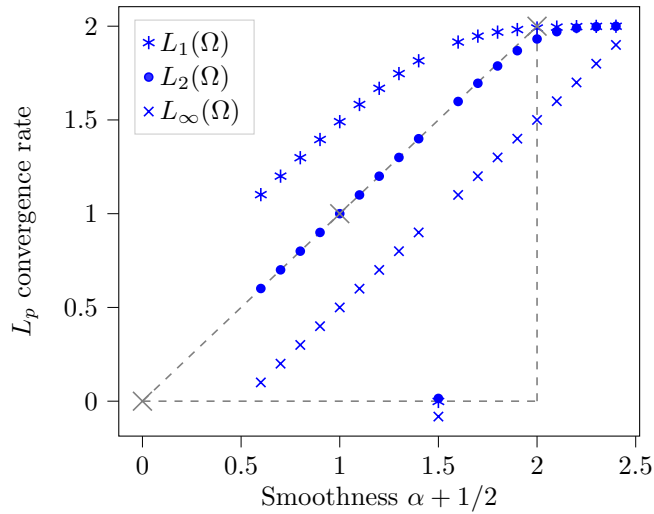

\subsection{Boundary conditions}\label{sec:ex_bcs}
We consider the Sobolev space $\calh(\Omega)\coloneqq W_2^2(0, 1)$ with its standard inner product. The reproducing kernel of this space is explicitly given
in point 2 of Lemma 4 in~\cite{saitoh2003reproducing}, and~\Cref{prop:bvp} with $d=1$ and $m=2$ states that $T(L_2(0,1))$ is the space of solutions 
of the PDE
\begin{equation}\label{eq:pde_example}
\begin{cases}
\begin{aligned}
u^{(4)}(x)-u^{(2)}(x)+ u(x) &=f(x),\;\;x\in(0, 1),\\
u^{(2)}(x) &= 0, \;\; x\in\{0, 1\},\\
u^{(3)}(x)- u^{(1)}(x) &= 0, \;\; x\in\{0, 1\}.
\end{aligned}
\end{cases}
\end{equation}
For $\alpha>0$ and $i=1,2,3$ we define three functions $f_{i,\alpha}(x)\coloneqq x^\alpha + p_i(x)$, where $p_i$ is a
suitable polynomial of degree $5$, so that $f_{1,\alpha}, f_{2,\alpha}, f_{3,\alpha}\in W_2^{\tau}(0,1)$ for all $\tau<\alpha+1/2$. In 
particular, they are all elements of $\calh(\Omega)$ if $\alpha>3/2$.
The polynomial terms are chosen to ensure, for a sufficiently large $\alpha$, that $f_{3,\alpha}$ satisfies all boundary conditions of~\eqref{eq:pde_example}, 
$f_{2,\alpha}$ satisfies only
those of order $2$, and $f_{1, \alpha}$ satisfies none. This implies that $f_{3,\alpha}\in T(L_2(\Omega))$ for $\alpha>4-1/2=7/2$.
To satisfy these requirements we use the functions
\begin{align*}
f_{1,\alpha}(x)&\coloneqq x^\alpha + x^2, \;\;
f_{2,\alpha}(x)\coloneqq x^\alpha + \frac{\alpha (1 - \alpha)}6 x^3,\\
f_{3,\alpha}(x)&\coloneqq x^\alpha +
\alpha x\left(\frac{-1 + 5 \alpha - 2 \alpha^2}{52}\ x^3
+\frac{8 - 14 \alpha + 3 \alpha^2}{39}\ x^2
+\frac{16 - 28 \alpha + 6 \alpha^2}{13}\right).
\end{align*}
For $f\in \calh(\Omega)$ we expect convergence as $h_X^{2 - 1/2}=h_X^{3/2}$ in the $L_\infty$ norm, and as $h_X^{2}$ in the $L_2$ norm, while for $f\in T(L_2(\Omega))$ 
according to \Cref{cor:application_to_sobolev} we
get $h_X^{7/2}$ ($L_\infty$ norm) and $h_X^{4}$ ($L_2$ norm).

We compute these rates for $20$ equally spaced values of $\alpha\in[3/2, 6]$. 
Figure~\ref{fig:bcs} reports these rates for $L_\infty$ (left) $L_2$ (right), as a function of the Sobolev smoothness $\alpha+1/2$. For both norms we get 
precisely the estimated rates for $f_{3,\alpha}\in\calh(\Omega)$ (i.e., $\alpha>3/2$) and $f_{3,\alpha}\in T(L_2(\Omega))$ (i.e., $\alpha>7/2$),
which are marked as crosses in the figure. Similarly, we obtain the same scaling for the 
intermediate $\alpha$ as predicted by \Cref{cor:application_to_sobolev} (dashed diagonal lines).
Also in this case, for $\alpha>7/2$ we observe a further increases of the $L_\infty$ rate by 
a factor $1/2$ before saturating, while the $L_2$ rate is limited by the value $4$, up to some minor oscillation around $\alpha=4$ that are likely 
due to an imprecise estimate of the rate.
For $f_{1, \alpha}$ and $f_{2, \alpha}$, in view of~\Cref{rem:boundary_interpolation} we expect that the converge rates scale with their smoothness (i.e., with 
$\alpha$) but also that they saturate after they stop meeting the boundary conditions. This can indeed be observed for both 
functions and both norms (Figure~\ref{fig:bcs}), even if also in this case we observe an increase of $1/2$ in the maximal rates.

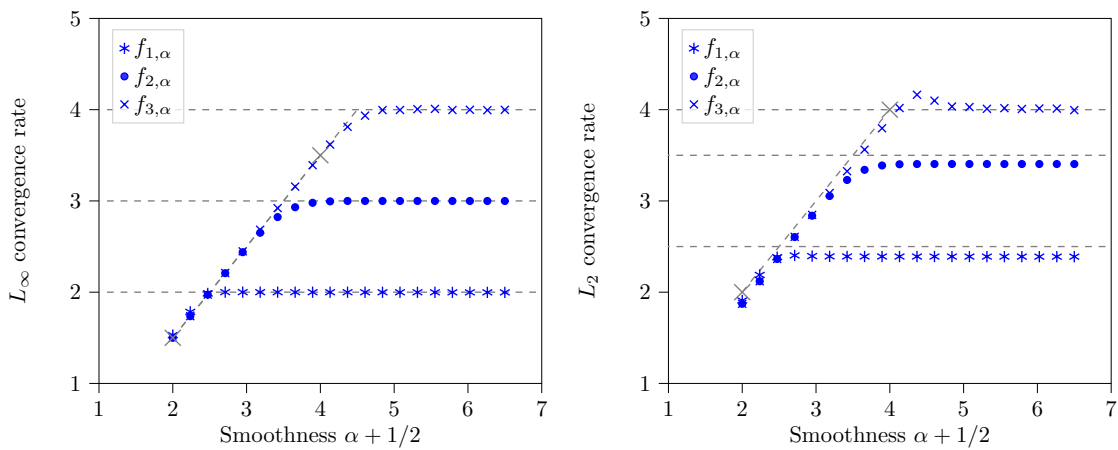
\begin{figure}[htp!]
\centering
\scalebox{.85}{
\begin{tikzpicture}

\definecolor{darkgray176}{RGB}{176,176,176}
\definecolor{gray}{RGB}{128,128,128}
\definecolor{lightgray204}{RGB}{204,204,204}

\begin{axis}[
legend cell align={left},
legend style={
  fill opacity=0.8,
  draw opacity=1,
  text opacity=1,
  at={(0.03,0.97)},
  anchor=north west,
  draw=lightgray204
},
tick align=outside,
tick pos=left,
x grid style={darkgray176},
xlabel={Smoothness $\alpha+1/2$},
xmin=1, xmax=7,
xtick style={color=black},
y grid style={darkgray176},
ylabel={$L_\infty$ convergence rate},
ymin=1, ymax=5,
ytick style={color=black},
ytick={0,1,2,3,4,5,6}
]
\addplot [semithick, blue, mark=asterisk, mark size=2.5, mark options={solid}, only marks]
table {%
2.0 1.530074774015412
2.2368421052631579649 1.7844677436177965
2.4736842105263157077 1.9858262123654398
2.7105263157894734505 2.000473587769605
2.9473684210526314153 2.0003137771811743
3.1842105263157893802 2.000148345805839
3.421052631578947345 1.9999808913903332
3.65789473684210531 1.99981287476013
3.8947368421052628307 1.9996449191578902
4.1315789473684207955 1.999477245024923
4.3684210526315787604 1.999311016370253
4.605263157894736281 1.9991473074557347
4.84210526315789469 1.9989845171760796
5.078947368421053099 1.9988229135012383
5.31578947368421062 1.9986634483725054
5.5526315789473681406 1.9985029307988778
5.7894736842105256613 1.998344657127824
6.0263157894736840703 1.9981882111664568
6.263157894736841591 1.9980314238476204
6.5 1.997876687434589
};
\addlegendentry{$f_{1,\alpha}$}
\addplot [semithick, blue, mark=*, mark size=1.5, mark options={solid}, only marks]
table {%
2.0 1.499976635630122
2.2368421052631579649 1.7367320899231218
2.4736842105263157077 1.9731733221154677
2.7105263157894734505 2.2082572646955088
2.9473684210526314153 2.438126608427544
3.1842105263157893802 2.651094050716503
3.421052631578947345 2.823186595918729
3.65789473684210531 2.9309980728156795
3.8947368421052628307 2.9793544717406
4.1315789473684207955 2.995085916586112
4.3684210526315787604 2.999741527493666
4.605263157894736281 2.999701321030858
4.84210526315789469 2.999663156424528
5.078947368421053099 2.999618923618138
5.31578947368421062 2.999575950414473
5.5526315789473681406 2.999535120364263
5.7894736842105256613 2.9994934404981293
6.0263157894736840703 2.9994589690409432
6.263157894736841591 2.9994152218648127
6.5 2.9993740654447953
};
\addlegendentry{$f_{2,\alpha}$}
\addplot [semithick, blue, mark=x, mark size=2.5, mark options={solid}, only marks]
table {%
2.0 1.499999877343951
2.2368421052631579649 1.736842616217672
2.4736842105263157077 1.9736857919228556
2.7105263157894734505 2.2105267375077906
2.9473684210526314153 2.4473675244697257
3.1842105263157893802 2.6842034364811567
3.421052631578947345 2.9209703469368677
3.65789473684210531 3.1573587284949363
3.8947368421052628307 3.39185333118752
4.1315789473684207955 3.6176486501349925
4.3684210526315787604 3.8118631859845546
4.605263157894736281 3.9344750159836117
4.84210526315789469 3.995665940510811
5.078947368421053099 3.9969212677775157
5.31578947368421062 4.005556364329684
5.5526315789473681406 4.008742066109757
5.7894736842105256613 3.9964321785394143
6.0263157894736840703 3.9988350294532133
6.263157894736841591 3.9946637159945517
6.5 3.9985520388147795
};
\addlegendentry{$f_{3,\alpha}$}
\addplot [semithick, gray, dashed, mark size=5, mark options={solid}, forget plot]
table {%
2 1.5
4.5 4
};
\addplot [semithick, gray, dashed, mark=x, mark size=5, mark options={solid}, forget plot]
table {%
2 1.5
4 3.5
};
\addplot [semithick, gray, dashed, forget plot]
table {%
0 2
10 2
};
\addplot [semithick, gray, dashed, forget plot]
table {%
0 3
10 3
};
\addplot [semithick, gray, dashed, forget plot]
table {%
0 4
10 4
};
\end{axis}

\end{tikzpicture}}
\scalebox{.85}{
\begin{tikzpicture}

\definecolor{darkgray176}{RGB}{176,176,176}
\definecolor{gray}{RGB}{128,128,128}
\definecolor{lightgray204}{RGB}{204,204,204}

\begin{axis}[
legend cell align={left},
legend style={
  fill opacity=0.8,
  draw opacity=1,
  text opacity=1,
  at={(0.03,0.97)},
  anchor=north west,
  draw=lightgray204
},
tick align=outside,
tick pos=left,
x grid style={darkgray176},
xlabel={Smoothness $\alpha+1/2$},
xmin=1, xmax=7,
xtick style={color=black},
y grid style={darkgray176},
ylabel={$L_2$ convergence rate},
ymin=1, ymax=5,
ytick style={color=black},
ytick={0,1,2,3,4,5,6}
]
\addplot [semithick, blue, mark=asterisk, mark size=2.5, mark options={solid}, only marks]
table {%
2.0 1.9092143803853547
2.2368421052631579649 2.188987149183617
2.4736842105263157077 2.383658857216741
2.7105263157894734505 2.4028797263595094
2.9473684210526314153 2.395331582975044
3.1842105263157893802 2.3923926103666875
3.421052631578947345 2.391515136707472
3.65789473684210531 2.3911862087542777
3.8947368421052628307 2.390979987872561
4.1315789473684207955 2.3907972558027213
4.3684210526315787604 2.3906199885334747
4.605263157894736281 2.3904435924317
4.84210526315789469 2.390269225688941
5.078947368421053099 2.3900953845758224
5.31578947368421062 2.3899233012676624
5.5526315789473681406 2.389751109206135
5.7894736842105256613 2.389580787023767
6.0263157894736840703 2.3894123983405926
6.263157894736841591 2.389244390692401
6.5 2.3890771196681357
};
\addlegendentry{$f_{1,\alpha}$}
\addplot [semithick, blue, mark=*, mark size=1.5, mark options={solid}, only marks]
table {%
2.0 1.8724790657764865
2.2368421052631579649 2.119818214744521
2.4736842105263157077 2.3637766700189844
2.7105263157894734505 2.604284260928471
2.9473684210526314153 2.8380523921121856
3.1842105263157893802 3.054110520268306
3.421052631578947345 3.2297378608870897
3.65789473684210531 3.34067790779139
3.8947368421052628307 3.3883825405666443
4.1315789473684207955 3.401664566011017
4.3684210526315787604 3.4042194848496052
4.605263157894736281 3.4045603962916102
4.84210526315789469 3.4045611270438285
5.078947368421053099 3.404512541657412
5.31578947368421062 3.404470788376715
5.5526315789473681406 3.404427838714633
5.7894736842105256613 3.404382331380942
6.0263157894736840703 3.4043500650381713
6.263157894736841591 3.4043045412912316
6.5 3.4042679891607084
};
\addlegendentry{$f_{2,\alpha}$}
\addplot [semithick, blue, mark=x, mark size=2.5, mark options={solid}, only marks]
table {%
2.0 1.8725047802954833
2.2368421052631579649 2.1199390952057016
2.4736842105263157077 2.3643248936825705
2.7105263157894734505 2.6066522886491077
2.9473684210526314153 2.8474682519582912
3.1842105263157893802 3.0870504734131345
3.421052631578947345 3.3254537265841297
3.65789473684210531 3.5623802995044858
3.8947368421052628307 3.796187971688326
4.1315789473684207955 4.017873065317524
4.3684210526315787604 4.163568374879332
4.605263157894736281 4.098981360666706
4.84210526315789469 4.034934771102358
5.078947368421053099 4.028820584714819
5.31578947368421062 4.00952609606462
5.5526315789473681406 4.016590688527335
5.7894736842105256613 4.007436685331195
6.0263157894736840703 4.013289803063447
6.263157894736841591 4.011975852556432
6.5 3.994999245087189
};
\addlegendentry{$f_{3,\alpha}$}
\addplot [semithick, gray, dashed, mark=x, mark size=5, mark options={solid}, forget plot]
table {%
2 2
4 4
};
\addplot [semithick, gray, dashed, forget plot]
table {%
0 2.5
10 2.5
};
\addplot [semithick, gray, dashed, forget plot]
table {%
0 3.5
10 3.5
};
\addplot [semithick, gray, dashed, forget plot]
table {%
0 4
10 4
};
\end{axis}

\end{tikzpicture}}
\caption{Convergence rates of the interpolation error of $f_{i,\alpha}$ for $i=1, 2, 3$, plotted as functions of $\alpha\in(3/2, 6]$, for the $L_\infty$
norm (left) and the $L_2$ norm (right). The horizontal dotted lines report the observed saturating values, while the theoretically expected rates are marked with 
crosses ($\calh(\Omega)$ and $T(L_2(\Omega))$), and with diagonal dashed lines (for the intermediate spaces).}\label{fig:bcs}
\end{figure}

\subsection{Periodic kernels} \label{sec:periodic-numsim}

Recall from Example~\ref{ex:periodic_kernel} that $W_{2,\textup{per}}^\alpha(\Omega)$, the periodic Sobolev space of order $\alpha > 0$ on $\Omega = [0, 1]$, consists of functions $f \colon [0, 1] \to \C$ of the form
\begin{equation} \label{eq:periodic-f-numsim}
  f(x) = \sum_{j \in \Z} c_j e^{2\pi \mathrm{i} j x} = \sum_{j \in \Z} c_j [ \cos(2\pi j x) + \mathrm{i} \sin(2\pi jx)]
\end{equation}
such that the coefficients $c_j \in \C$ satisfy $\sum_{j \in \Z} \lvert j \rvert^{2\alpha} \lvert c_j \rvert^2 < \infty$.
Here we consider the periodic kernel $k_r$ in~\eqref{eq:periodic-kernel} for $r \in \N$.
The RKHS of this kernel is $W_{2,\textup{per}}^r(\Omega)$ and, as observed in Example~\ref{ex:periodic_kernel}, the $\theta$th power of the RKHS is simply $\calh_\theta(k_r, \Omega) = W_{2,\textup{per}}^{\theta r}(\Omega)$.

We let $\xi_j \in \{-1, 0, 1\}$ be independent uniformly distributed random integers and consider functions
\begin{equation} \label{eq:periodic-sample-funcs}
  f_\alpha(x) = 1 + \sum_{j = 1}^\infty j^{-\alpha} \xi_j \cos(2\pi j x),
\end{equation}
which correspond to the selection $c_0 = 1$ and $c_j = c_{-j} = \frac{1}{2} \lvert j \rvert^{-\alpha} \xi_{\lvert j \rvert}$ for $j \neq 0$ 
in~\eqref{eq:periodic-f-numsim}.
Therefore \smash{$f_\alpha \in W_{2,\textup{per}}^\tau(\Omega)$} with probability one if $\tau < \alpha + 1/2$ and with probability zero if $\tau \geq \alpha + 
1/2$.
That is,
\begin{equation*}
  \mathbb{P}[ f_\alpha \in \calh_\theta(k_r, \Omega) ] = 1 \: \text{ if } \: \theta r < \alpha + 1/2 \quad \text{ and } \quad \mathbb{P}[ f_\alpha \in \calh_\theta(k_r, \Omega) ] = 0 \: \text{ if } \: \theta r \geq  \alpha + 1/2.
\end{equation*}
Some $f_\alpha$ are displayed in \Cref{fig:periodic_functions}.
In computations we truncate the series~\eqref{eq:periodic-sample-funcs} after 1,000 terms.
Because $\calh(k_r, \Omega) = W_{2,\textup{per}}^r(\Omega)$ is a sub-space of $W_2^r(\Omega)$, by \Cref{th:std_error_bound} we expect at least the rate \smash{$h_X^{r-1/2}$} in the $L_\infty$ norm and rate $h_X^r$ in the $L_1$ and $L_2$ norms for $f_\alpha \in W_{2,\textup{per}}^r(\Omega)$.
By \Cref{cor:application_to_sobolev} and \smash{$\calh_\theta(k_r, \Omega) = W_{2,\textup{per}}^{\theta r}(\Omega)$}, for $f_\alpha \in W_{2,\textup{per}}^{\theta r}(\Omega)$ we expect at least the $L_\infty$ rate \smash{$h_X^{(1+\theta) r - 1/2}$} and the $L_1$ and $L_2$ rate $h_X^{(1+\theta)r}$ if $\theta \in [0, 1]$.

Results are reported in \Cref{fig:periodic_rates} for $r = 1$ and $r = 2$.
We observe an overall agreement with the theory. For $r = 1$ saturation is evident after $\alpha = 3/2$, as predicted by the theory, while for $r = 2$ the $L_1$ and $L_2$ rates appear to saturate only after $\alpha = 4$, rather than $\alpha = 7/2$ predicted by the theory. 
Moreover, note that the $L_\infty$ rate appears to more or less match the $L_1$ and $L_2$ rates for all $\alpha$ when $r = 1$ but saturates already after $\alpha = 7/2$ when $r = 2$.
We lack an explanation for these phenomena.

\begin{figure}[th]
  \centering
  \includegraphics[width=\textwidth]{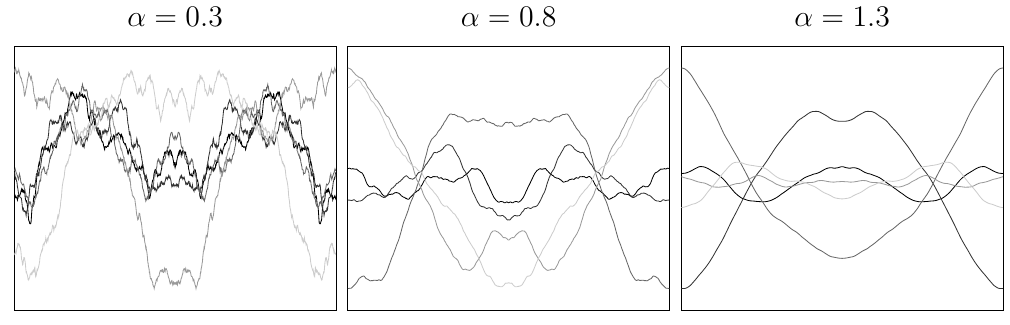}
  \caption{Five functions $f_\alpha$ on $[0, 1]$ sampled according to~\eqref{eq:periodic-sample-funcs} for each $\alpha \in \{0.3, 0.8, 1.3\}$.}
  \label{fig:periodic_functions}
\end{figure}

\begin{figure}[th]
  \centering
  \setlength\fwidth{.48\textwidth}
\begin{tikzpicture}

\definecolor{darkgray176}{RGB}{176,176,176}
\definecolor{gray}{RGB}{128,128,128}
\definecolor{lightgray204}{RGB}{204,204,204}

\begin{axis}[
width=0.951\fwidth,
height=0.75\fwidth,
at={(0\fwidth,0\fwidth)},
legend cell align={left},
legend style={
  fill opacity=0.8,
  draw opacity=1,
  text opacity=1,
  legend pos = south east,
  draw=lightgray204
},
tick align=outside,
tick pos=left,
x grid style={darkgray176},
xmin=0.0, xmax=2.5,
xtick style={color=black},
y grid style={darkgray176},
ymin=0.0, ymax=2.5,
ytick style={color=black},
xlabel = {Smoothness $\alpha + 1/2$},
ylabel = {Convergence rate},
title = {$r = 1$}
]

\addplot
  [
    semithick, 
    blue, 
    mark=asterisk, 
    mark size=2.5,  
    mark options={solid}, 
    only marks
  ] 
  table [x=r_f, y=L1]{Figures/data/rates-r=1.txt};

\addlegendentry{$L_1(\Omega)$}
\addplot [semithick, blue, mark=*, mark size=1.5, mark options={solid}, only marks]
table [x=r_f, y=L2]{Figures/data/rates-r=1.txt};
\addlegendentry{$L_2(\Omega)$}
\addplot [semithick, blue, mark=x, mark size=2.5, mark options={solid}, only marks]
table [x=r_f, y=L_inf]{Figures/data/rates-r=1.txt};
\addlegendentry{$L_\infty(\Omega)$}
\addplot [semithick, gray, dashed, mark=x, mark size=5, mark options={solid}, forget plot]
table {%
0.2 0.2
2.3 2.3 
};
\addplot [semithick, gray, dashed, forget plot]
table {%
0 2.0
2.5 2.0
};
\end{axis}

\end{tikzpicture}
\begin{tikzpicture}

\definecolor{darkgray176}{RGB}{176,176,176}
\definecolor{gray}{RGB}{128,128,128}
\definecolor{lightgray204}{RGB}{204,204,204}

\begin{axis}[
width=0.951\fwidth,
height=0.75\fwidth,
at={(0\fwidth,0\fwidth)},
legend cell align={left},
legend style={
  fill opacity=0.8,
  draw opacity=1,
  text opacity=1,
  legend pos = south east,
  draw=lightgray204
},
tick align=outside,
tick pos=left,
x grid style={darkgray176},
xmin=0.0, xmax=5.0,
xtick style={color=black},
y grid style={darkgray176},
ymin=0.0, ymax=5.0,
ytick style={color=black},
xlabel = {Smoothness $\alpha + 1/2$},
title = {$r = 2$}
]

\addplot
  [
    semithick, 
    blue, 
    mark=asterisk, 
    mark size=2.5,  
    mark options={solid}, 
    only marks
  ] 
  table [x=r_f, y=L1]{Figures/data/rates-r=2.txt};

\addlegendentry{$L_1(\Omega)$}
\addplot [semithick, blue, mark=*, mark size=1.5, mark options={solid}, only marks]
table [x=r_f, y=L2]{Figures/data/rates-r=2.txt};
\addlegendentry{$L_2(\Omega)$}
\addplot [semithick, blue, mark=x, mark size=2.5, mark options={solid}, only marks]
table [x=r_f, y=L_inf]{Figures/data/rates-r=2.txt};
\addlegendentry{$L_\infty(\Omega)$}
\addplot [semithick, gray, dashed, mark=x, mark size=5, mark options={solid}, forget plot]
table {%
0.2 0.2
4.8 4.8 
};
\addplot [semithick, gray, dashed, forget plot]
table {%
0 4.0
5 4.0
};
\end{axis}

\end{tikzpicture}
    \caption{Convergence rates of the interpolation error of $f_{\alpha}$ in~\eqref{eq:periodic-sample-funcs} for the periodic kernels of orders $r = 1$ (left) and $r = 2$ (right) given in~\eqref{eq:periodic-kernel}. For each $\alpha$, the shown rates are averages over rates for 100 independent $f_\alpha$.}
  \label{fig:periodic_rates}
\end{figure}
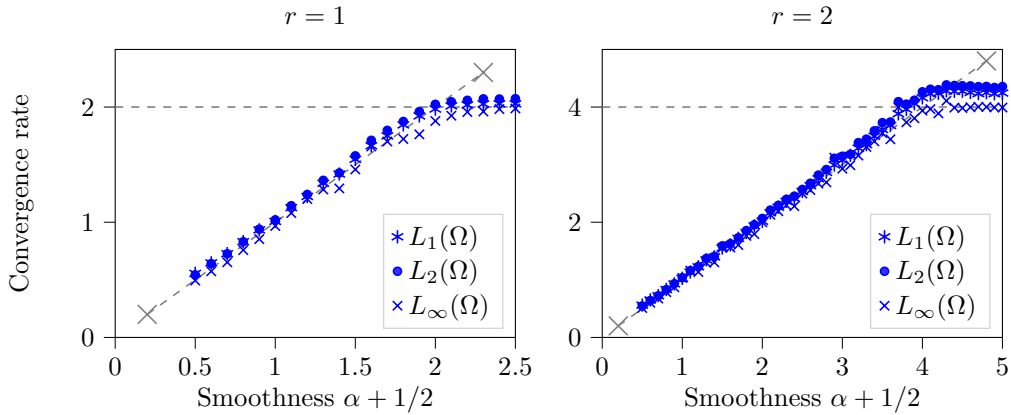

\section{Conclusion and outlook}
This work extends the theory of superconvergence of kernel approximation by identifying broader conditions, formulated via operator ranges and interpolation spaces, under which improved approximation rates can be guaranteed. In particular, we show that for a wide class of kernels and function spaces, including Sobolev RKHSs, superconvergence emerges from additional regularity and structural properties of the target function.

Several questions have emerged from our analysis and remain so far unanswered. 
First, the numerical experiments show a better rate of convergence in the $L_1$ than in the $L_2$ norm, and a saturation of these rates to values that are 
larger than predicted by a $1/2$ term.
Moreover, our understanding of the interpolation spaces for some operators is still limited.
We also believe that much more could be said about the periodic setting considered \Cref{ex:periodic_kernel} and \Cref{sec:periodic-numsim}.

In the Sobolev setting, we still lack tools to determine an exact association between boundary conditions and a given kernel. Moving from boundary conditions 
to a kernel would be helpful for example in adapting the results of this paper to the solution of PDEs by collocation, while knowing the exact boundary 
conditions of some commonly used kernels, such as the Mat\`ern and Wendland ones, would help to improve the convergence theory for certain classes of functions. 
Further connections to the works~\cite{fasshauer2012green,fasshauer2011reproducing,fasshauer2013reproducing} should be explored in this direction.

Apart from addressing these questions, we plan to consider the case of general (non orthogonal) projections, which would cover even more general approximation 
schemes.
Moreover, in view of the inverse statements as derived in \cite{wenzel2025sharp},
we also aim at deriving corresponding inverse statements for the general superconvergence case.

\section*{Acknowledgement}
The authors would like to thank Robert Schaback for commenting on an early version of this paper, and Luigi Provenzano for providing useful comments on the 
topic of Section~\ref{sec:boundary}. 

T.K.\ was supported by the Research Council of Finland grant 359183 (Flagship of Advanced Mathematics for Sensing, Imaging and Modelling).
G.S.\ is a member of INdAM-GNCS, and his
work was partially supported by the project ``Perturbation problems and asymptotics for elliptic differential equations: variational and potential theoretic 
method'' funded
by the program “NextGenerationEU” and by MUR-PRIN, grant 2022SENJZ3.

\bibliography{references_paper}
\bibliographystyle{abbrv}
\end{document}